\documentclass[a4paper]{article}
\usepackage{a4wide}
\usepackage{authblk}
\providecommand{\keywords}[1]{\smallskip\noindent\textbf{\emph{Keywords:}} #1}

\usepackage{graphicx}
\usepackage{tikz}
\usetikzlibrary{fit,patterns,decorations.pathmorphing,decorations.pathreplacing,calc}

\usepackage{amsmath,amssymb,amsthm}
\usepackage[inline]{enumitem}

\usepackage{thmtools}

\usepackage{hyperref}
\usepackage[capitalize,nameinlink,sort]{cleveref}

\usepackage[T1]{fontenc}
\usepackage{newpxtext}
\usepackage{euler}

\usepackage[textsize=scriptsize]{todonotes}

\newcommand{\N}{\mathbb{N}}

\DeclareMathOperator{\Fac}{Fac}
\DeclareMathOperator{\val}{val}
\DeclareMathOperator{\rep}{rep}

\DeclareMathOperator{\pref}{pref}

\newcommand{\smatrix}[1]{
	\left(\begin{smallmatrix} #1 \end{smallmatrix} \right)
}

\newcommand{\infw}[1]{\mathbf{#1}}

\declaretheorem[numberwithin=section]{theorem}
\declaretheorem[sibling=theorem]{lemma,corollary,proposition}
\declaretheorem[sibling=theorem,style=definition]{example,definition,remark,observation}

\declaretheorem{claim}

\declaretheoremstyle[
    headfont=\normalfont\itshape, 
    bodyfont = \normalfont,
    qed=$\blacksquare$, 
    headpunct={:}]{claimproofstyle} 
\declaretheorem[name={Proof of claim}, style=claimproofstyle, unnumbered]{claimproof}

\crefname{equation}{}{}

\title{On extended boundary sequences of morphic and Sturmian words}

\author{Michel Rigo}

\author{Manon Stipulanti\thanks{Supported by the FNRS Research grant 1.B.397.20F.}}

\author{Markus A. Whiteland\thanks{Supported by the FNRS Research grant 1.B.466.21F.}}

\affil{Department of Mathematics, University of Li\`ege, Li\`ege, Belgium}

\affil{\texttt{\{m.rigo,m.stipulanti,mwhiteland\}@uliege.be}}

\date{}

\begin{document}

\maketitle

\begin{abstract}
  Generalizing the notion of the boundary sequence introduced by Chen and Wen, the $n$th term of the $\ell$-boundary sequence of an infinite word is the finite set of pairs $(u,v)$ of prefixes and suffixes of length $\ell$ appearing in factors $uyv$ of length $n+\ell$ ($n\ge \ell\ge 1$). Otherwise stated, for increasing values of $n$, one looks for all pairs of factors of length $\ell$ separated by $n-\ell$ symbols.

  For the large class of addable abstract numeration systems $S$, we show that if an infinite word is $S$-automatic, then the same holds for its $\ell$-boundary sequence. In particular, they are both morphic (or generated by an HD0L system). To precise the limits of this result, we discuss examples of non-addable numeration systems and $S$-automatic words for which the boundary sequence is nevertheless $S$-automatic and conversely, $S$-automatic words with a boundary sequence that is not $S$-automatic. In the second part of the paper, we study the $\ell$-boundary sequence of a Sturmian word. We show that it is obtained through a sliding block code from the characteristic Sturmian word of the same slope.
We also show that it is the image under a morphism of some other characteristic Sturmian word.
  \end{abstract}
  
  \keywords{Boundary sequences, Sturmian words, Numeration systems, Automata, Graph of addition}

\section{Introduction}
Let $\mathbf{x}$ be an infinite word, i.e., a sequence of letters belonging to a finite alphabet. Imagine a window of size~$n$ moving along $\mathbf{x}$. Such a reading frame permits to detect all factors of length~$n$ occurring in~$\mathbf{x}$. For instance, the factor complexity function of $\mathbf{x}$ mapping $n\in\mathbb{N}$ to the number of distinct factors of length~$n$ is extensively studied in combinatorics on words. Now let $n,\ell$ be such that $n\ge \ell$. Assume that within the sliding window, we only focus on its first and last $\ell$ symbols. Otherwise stated, for a factor $uyv$ of length $n$, we only consider its borders $u$ and $v$ of length~$\ell$.

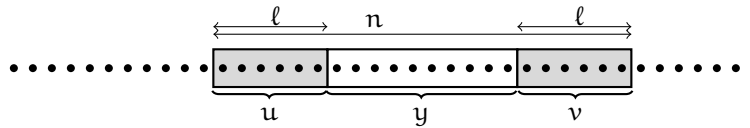
\begin{figure}[h!t]
  \begin{center}
  \begin{tikzpicture}
  \node [fit={(4.5,1) (7,1.5)}, inner sep=0pt, draw=black, thick] (x) {};
  \node [fit={(3,1) (4.5,1.5)}, inner sep=0pt, draw=black, thick, fill=gray!30] (u) {};
  \node [fit={(7,1) (8.5,1.5)}, inner sep=0pt, draw=black, thick, fill=gray!30] (v) {};
  \draw [thick, decoration={brace, mirror, raise=0.3cm}, decorate] (u.west) -- node[below=.4cm] {$u$} (x) {};
  \draw [thick, decoration={brace, mirror, raise=0.3cm}, decorate] (x.east) -- node[below=.4cm] {$v$} (v.east) ;
  \draw [thick, decoration={brace, mirror, raise=0.3cm}, decorate] (x.west) -- node[below=.4cm] {$y$} (v.west) ;
  \draw[<->] (3,1.7) -- (8.5,1.7) ;
  \node [fit={(5,1.8)}] (t) {$n$} ;
\draw[<->] (7,1.8) -- (8.5,1.8) ;
  \node [fit={(7.75,1.82)}] (t) {$\ell$} ;
  \draw[<->] (3,1.8) -- (4.5,1.8) ;
  \node [fit={(3.75,1.82)}] (t) {$\ell$} ;

  \foreach \r in {1,...,40}{
    \node[circle,fill=black,inner sep=0pt,minimum size=3pt] (a) at (.125+.25*\r,1.25) {};
    }
\end{tikzpicture}
\end{center}
\caption{A sliding window where we focus on two regions of a fixed length.}
  \label{fig:frame}
\end{figure}

For any given window length~$n$, we would like to determine what are the pairs of length-$\ell$ borders that may occur. This leads to the following definition, where, to simplify notation, we consider borders of factors of length $n + \ell$ rather than $n$.  

\begin{definition}
Let $\ell\in\mathbb{N}_{>0}$ and $\infw{x} \in A^{\N}$. For $n\ge \ell$, we define the
\emph{$n$th boundary set} by
\begin{equation*}
\partial_{\mathbf{x},\ell}[n] := \{(u,v)\in A^\ell\times A^\ell\mid uyv \text{ is a factor of }\mathbf{x} \text{ for some } y\in A^{n-\ell}\}
\end{equation*}
and call the sequence $\partial_{\mathbf{x},\ell} := (\partial_{\mathbf{x},\ell}[n])_{n\ge \ell}$ the {\em $\ell$-boundary sequence} of $\mathbf{x}$. When $\ell=1$, we write $\partial_{\mathbf{x},1}=\partial_\mathbf{x}$ and simply talk about the {\em boundary sequence}.
\end{definition}

The $\ell$-boundary sequence takes values in $2^{A^\ell\times A^\ell}$, and hence itself can be seen as an infinite word over a finite alphabet. We give an introductory example.

\begin{example}\label{exa:intro}
Consider the Fibonacci word $\mathbf{f}=0100101001\cdots$; the fixed point of the morphism $0\mapsto 01, 1\mapsto 0$.
We have $\partial_{\mathbf{f}} = a\, b\, b\, a\, b\, b\, b\, b\, a\, b\, b\, a\, b\, b\, b\, b\, a\, b\, b\, b\, b\, a\, b\, b\, a\, b\, b\, b\, b\cdots$,
where $a:=\{(0,0), (0,1), (1,0)\}$ and $b:=\{0,1\} \times \{0,1\}$.
For instance, $\partial_{\mathbf{f}}[1]=a$ because the length-$2$ factors of $\mathbf{f}$ are $00, 01, 10$, while $\partial_{\mathbf{f}}[2]=b$ because its length-$3$ factors are of the form $0\_0, 0 \_ 1, 1 \_ 0, 1 \_ 1$ (they are in fact $010, 001, 100, 101$).
The $2$-boundary sequence starts with
\[
\partial_{\mathbf{f},2} = a\ b\ c \ d\ e\ f\ b\ c\ d\ b\ c\ d\ e\ f\ b\ c\ d\ e\ f\ b\ c\ d\ b\ c\ d\cdots
\]
where
\begin{align*}
a &:=  \{(00,10),(01,00),(01,01),(10,01),(10,10)\}, \\
b &:= \{(00,00),(00,01),(01,01),(01,10),(10,00),(10,10)\}, \\
c &:= \{(00,01),(00,10),(01,00),(01,10),(10,00),(10,01)\}, \\
d &:= \{(00,00),(00,10),(01,00),(01,01),(10,01),(10,10)\}, \\
e &:= \{(00,01),(01,01),(01,10),(10,00),(10,01),(10,10)\}, \\
f &:= \{(00,10),(01,00),(01,01),(01,10),(10,01),(10,10)\}.
\end{align*}

The first element $\partial_{\mathbf{f},2}[2]=a$ is peculiar; it
corresponds exactly to the five length-$4$ factors occurring in
$\mathbf{f}$. Our \cref{prop:first-letter} shows that $a$ appears only once in $\partial_{\infw{f}}$. Then, e.g.,
$\partial_{\mathbf{f},2}[3] = b$ because the length-$5$ factors of
$\mathbf{f}$ are of the form $00\_00$, $00\_01$, $01\_01$, $01\_10$,
$10\_00$ and $10\_10$ (the factors are $00100$, $00101$, $01001$, $10100$, $10010$, and $01010$). For
length-$6$ factors, note that two are of the form $10u01$ for some
$u\in\{0,1\} \times \{0,1\}$. All letters, except $a$, appear infinitely often
in $\partial_{\mathbf{f},2}$: see \cref{the:sturmian}.
\end{example}

\subsection{Motivation and related work}

In combinatorics on words, borders and boundary sets are related to important concepts.
For instance, a word $v$ is {\em bordered} if there
exist $u,x,y$ such that $v=ux=yu$ and $0<|u|<|v|$. One reason to
study bordered words is Duval's theorem: for a sufficiently long word $v$, the
maximum length of unbordered factors of $v$ is equal to the
period of $v$ \cite{Duval}.
In formal language theory, a language $L$ is
{\em locally $\ell$-testable} (LT) if the membership of a word
$w$ in $L$ only depends on the prefix, suffix and factors of
length~$\ell$ of $w$. In \cite{Place}, the authors consider the
so-called {\em separating problem} of languages by LT languages;
they utilize {\em $\ell$-profiles} of a word, which can again be
related to boundary sets. Let us also mention that, in bioinformatics and computational biology, one of the aims is to reconstruct sequences from subsequences \cite{Margaritis}. To determine DNA segments by bottom-up analysis, {\em paired-end} sequencing is used. In this case both ends of DNA fragments of known length are sequenced. See, for instance, \cite{Fullwood}. This is quite similar to the theoretical concept we discuss here.

The notion of a ($1$-)boundary sequence was introduced by  Chen and Wen in
\cite{ChenWen} and was further studied in \cite{Guo}, where it is
shown that the boundary sequence of a \emph{$k$-automatic}
word (in the sense of Allouche and Shallit \cite{AS}: see \cref{def:Saut-kaut}) is $k$-automatic.
It is well known that a $k$-automatic word $\infw{x}$ is
\emph{morphic}, i.e., there exist morphisms $f\colon A \to A^*$
and $g\colon A \to B$ and a letter $a \in A$ such that $\infw{x} = g(f^{\omega}(a))$,
where $f^{\omega}(a) = \lim_{n\to \infty} f^n(a)$. However, $k$-automatic words (with $k$ ranging over the integers) do not capture all morphic words: a well-known
characterization of $k$-automatic words is given by Cobham~\cite{Cobham1972uniform} (the generating morphism $f$ maps each letter to a length-$k$ word). 
This paper is driven by the natural
question whether, in general, the $\ell$-boundary sequence of a morphic word is morphic.
In case such generating morphisms can be constructed, we have at our disposal a simple algorithm providing the set of length-$\ell$ borders in factors of all lengths.

We briefly present several situations in which the notion of
boundary sets is explicitly or implicitly used. In
\cite[Thm.~4]{CurrieHOR2019squarefree}, the authors study the boundary sequence to exhibit a squarefree word for which each subsequence arising from an arithmetic progression contains a square.
Boundary sets play an important role in the study of so-called \emph{$k$-abelian} and \emph{$k$-binomial complexities} of infinite words (for definitions, see \cite{Rigo2017}).
For instance, computing the $2$-binomial complexity of generalized Thue--Morse words \cite{lu20212binomial} requires inspecting pairs of prefixes and suffixes of factors, which is again related to the boundary sequence when these prefixes and suffixes have equal length. The $k$-binomial complexities of images of binary words under powers of the
Thue--Morse morphism are studied in \cite{RigoSW2022binomial}; there some
general properties of boundary sequences of binary words are required (see~\cite[Lem.~4.6]{RigoSW2022binomial}).  
Moreover, if $\partial_{\mathbf{x}}$ is automatic, then the
abelian complexity of the image of $\mathbf{x}$ under a
so-called Parikh-constant morphism is automatic \cite{ChenWen}.
Guo, L\"u, and Wen combine this result with theirs in \cite{Guo} to
establish a large family of infinite words with automatic abelian
complexity.

Let $k\ge 1$. We let $\equiv_k$ denote the $k$-abelian
equivalence, i.e., $u \equiv_k v$ if the words $u$ and $v$ share the same set of factors of length at most $k$ with the same multiplicities \cite{KSZ}. For $u$ and $v$ equal length factors
of a Sturmian word $\mathbf{s}$, we have $u\equiv_k v$ if and only if they share a common prefix and a
common suffix of length $\min\{|u|, k-1\}$ and $u\equiv_1 v$ \cite[Prop.~2.8]{KSZ}. 
Under the assumption that the largest power of a letter appearing in 
$\infw{s}$ is less than $2k-2$,
the requirement $u \equiv_1 v$ in the previous result may be
omitted \cite[Thm.~3.6]{PeltoWhiteland} (compare to \cref{prop:first-letter}). Thus the quotient of the set of factors of
length $n$ occurring in a Sturmian word by the
relation~$\equiv_k$ is completely determined by
$\partial_{\mathbf{s},k-1}[n-k+1]$ for large enough~$k$
(depending on $\mathbf{s}$). Other families of words with
$k$-abelian equivalence determined by the boundary sets are given in \cite[Prop.~4.2]{PeltoWhiteland}.


\subsection{Our contributions}

Up to our knowledge, we are the first to propose a systematic study of the
$\ell$-boundary sequences of infinite words. It is therefore
natural to consider the notion on well-known classes of words.
In this paper, we consider morphic words and Sturmian words.

Any morphic word is $S$-automatic for some abstract numeration system $S$ \cite{Maes}.
With \cref{the:kernel}, we prove that for a large class of numeration systems $S$, if $\mathbf{x}$ is an $S$-automatic word, then the boundary sequence $\partial_\mathbf{x}$ is again $S$-automatic. Our approach generalizes the arguments provided by \cite{Guo}. Considering exotic numeration systems allows a better understanding of underlying mechanisms, which do not arise in the ordinary integer base systems. In particular, we deal with addition within the numeration system; in integer base systems, the carry propagation is easy to handle (by a two-state finite automaton). Our arguments apply to so-called \emph{addable} numeration systems for which the graph of addition is regular (see \cref{def:addable} for details).

As an alternative, we observe that a classical effective procedure (\cref{thm:walnut}) transforming formulae to automata can be extended to addable abstract numeration
systems $S$. The $S$-automaticity of the
$\ell$-boundary sequence then follows from the fact that it
is definable by a first-order formula of the structure $\langle \N,+ \rangle$
extended with comparisons and indexing into an $S$-automatic sequence.

This alternative proof however hides the important details that
might help identifying the technical limits of the result: not all morphic words allow an addable system to work with.
However, the finiteness of a suitable kernel captures all morphic words (see \cref{thm:Ukernel}). 
To identify the contours of our result, we also discuss the case where $\mathbf{x}$ is $S$-automatic and $\partial_\mathbf{x}$ is not $S$-automatic.  To construct such examples, we have to consider non-addable numeration systems in \cref{ss:ce}.

We then turn to the other class of words under study. Letting
$\mathbf{s}$ be a Sturmian word with slope $\alpha$, with \cref{the:sturmian} we show that the $\ell$-boundary sequence of $\mathbf{s}$ is obtained through a sliding block code
from the \emph{characteristic Sturmian word of slope $\alpha$} (see \cref{sec:Sturmian} for a definition) up to the first letter.
This result
holds even for non-morphic Sturmian words, so for an arbitrary irrational $\alpha$. Where the techniques used in the first part of the paper have an automata-theoretic flavor, the second part relies on the geometric characterization of Sturmian words as codings of rotations.
We provide another description of the $\ell$-boundary sequence of a
Sturmian word as the morphic image of some characteristic Sturmian
word in \cref{prop:another-charact}.

This paper is a long version of \cite{RigoSW2022extended} presented at MFCS 2022. It contains many proofs (omitted due to space limitation) and, in particular, discussions about Sturmian words. This extended version includes work through examples using {\tt Walnut}. In \cref{ss:logic} we explicitly compute the $2$-boundary sequence of the Thue--Morse and Fibonacci words, see \cref{exa:walnut1,exa:walnut2}. In \cref{ss:non-add}, we present several examples of automatic sequences built on intrinsically non-addable numeration systems for which the boundary sequence is still automatic, see \cref{pro:add1,pro:add2}. Finally, the proof of \cref{the:kernel} has been strengthened to a larger setting to include addable abstract numeration systems.  This slightly broadens the presentation of the paper which is not limited to positional numeration systems anymore.

\section{Preliminaries}
Throughout this paper we let $A$ denote a finite alphabet. Then $A^n$
denotes the set of length-$n$ words and $A^ {\N}$ denotes the set of infinite words. Infinite words will usually, but not always, be indexed starting from $0$. 
They will also be written in bold.
For a finite word $u$, we let $u^\omega$ denote the concatenation of infinitely many copies of the word $u$, i.e.,  $u^\omega = u u u \cdots$.
For two words $u,v$ for which $w=uv$,  we let $wv^{-1}$ denote the prefix $u$ and $u^{-1}w$ the suffix $v$. 
For a finite or infinite word
$\infw{x}$, we let $\mathbf{x}[n]$ denote the letter at index $n$
(assuming it is well-defined for this value of $n$, e.g., if $\infw{x}$ is a $\ell$-boundary sequence, $n\ge\ell$). 
Similarly, for $m\geq n$ we set $\infw{x}[n,m] := \infw{x}[n]\cdots\infw[m]$.
For any integer $n\ge 0$,  we let $\Fac_n(\infw{x})$ denote the set of length-$n$ factors of $\infw{x}$; we write $\Fac(\infw{x}) = \cup_{n\ge 0} \Fac_n(\infw{x})$.
A factor $u$ of an infinite word $\infw{x} \in A^{\N}$ is called
\emph{right special} if there exist distinct letters $a,b\in A$
such that $ua$, $ub \in \Fac(\infw{x})$. We note that an infinite word~$\infw{x}$
is aperiodic if and only if it has a right special factor for each length.
For general references on numeration systems, see \cite{Fraenkel} and \cite[Chap.~1--3]{cant2010}. We assume that the reader has some knowledge in automata theory. For a reference see \cite{ShallitAutomata} or \cite[Chap.~1]{RigoBook}.

\subsection{Basic properties of boundary sequences}

Recall that in our definition of the boundary sequence, we inspect factors of length $n + \ell$ with $n\geq \ell$.
This implies that the prefix and suffix of length $\ell$ forming the boundary pair do not overlap.
The following observation justifies this choice in a sense.
\begin{proposition}
Let $\infw{x}$ be an aperiodic word and $\ell \geq 1$ be an integer. Then the boundary set
$\partial_{\infw{x},\ell}[m]$, with~$0\le m < \ell$, appears exactly once in the sequence
$(\partial_{\infw{x},\ell}[n])_{n\geq 0}$.
\end{proposition}
\begin{proof}
Fix an integer $m$ with $0 \leq m < \ell$. We show
that~$\partial_{\infw{x},\ell}[m] \neq \partial_{\infw{x},\ell}[n]$ for any $n > m$.
The claim follows straightforwardly from this observation. We first observe that any
boundary pair~$(u_1 \cdots u_{\ell},v_1 \cdots v_{\ell})$ in~$\partial_{\infw{x},\ell}[m]$
satisfies $u_{m+1} \cdots u_{\ell} = v_1 \cdots v_{\ell - m}$. In particular,
$u_{\ell} = v_{\ell - m}$ for any pair in $\partial_{\infw{x},\ell}[m]$. Consider then the
boundary set~$\partial_{\infw{x},\ell}[n]$ with~$n > m$. Let
$x = x_1 \cdots x_{n + \ell}$ be a factor of length $n + \ell$ such that
$x_1 \cdots x_{n + \ell - m - 1}$ is right special and $x_{\ell} \neq x_{n + \ell - m}$ (here $\ell < n + \ell - m$ so such a choice can be made). Now $x$ defines the boundary pair $( x_1 \cdots x_{\ell} , x_{n+1} \cdots x_{n + \ell} ) = (u_1\cdots u_{\ell},v_1 \cdots v_{\ell})$ for which $u_{\ell} \neq v_{\ell - m}$,
which shows that this pair cannot appear in $\partial_{\infw{x},\ell}[m]$. This concludes the proof.
\end{proof}

The above proposition is tight in the sense that there exist aperiodic words for which the boundary set
$\partial_{\infw{x},\ell}[\ell]$ appears infinitely often in  the boundary sequence $\partial_{\infw{x},\ell}$.
This can be seen, e.g., from \cref{prop:first-letter}. Another quick example for this is
the Champernowne word $\infw{c} = 0\, 1\, 00\, 01\, 10\, 11\cdots $ (the concatenation of
the radix-ordered binary representations of the naturals) for which
$\partial_{\infw{c},\ell} = (\{0,1\}^{\ell} \times \{0,1\}^{\ell})^{\omega}$.

\begin{lemma}
For any $\ell \geq 1$, the $\ell$-boundary sequence of an eventually periodic word is eventually periodic.
\end{lemma}
\begin{proof}
Let $\infw{x} = uv^{\omega}$. We claim that $\partial_{\infw{x},\ell}[n + |v|] = \partial_{\infw{x},\ell}[n]$
for all $n\geq \max\{\ell, |u|\}$. Indeed, consider a factor $x$ of length $ n + |v| + \ell$ occurring at
position $i$. We may write $x = x's$ with $|x'| = n + |v|$ and $|s|=\ell$. Since $n\geq |u|$, there exists a factorization
$v = v_1v_2$ such that $x'$ ends with $v v_1$, and $s$ is a prefix of $(v_2v_1)^{\omega}$. The factor of length
$n + \ell$ occurring at position $i$ is thus $x'(v_2v_1)^{-1}s$. We have shown that the boundary pairs
$(\infw{x}[i, i+\ell-1], \infw{x}[i + n,i + n + \ell-1])$ and $(\infw{x}[i, i+\ell-1], \infw{x}[i + n + |v|,i + n + |v| + \ell-1])$
are equal. This suffices for the proof.
\end{proof}

\subsection{Numeration systems and automatic words}

For general references about automatic words and abstract numeration systems, see \cite{AS} and \cite{Maes} or \cite[Chap.~3]{cant2010}. An \emph{abstract numeration system} (ANS) 
is a triple $S = (L,A,<)$ with $L$ an infinite regular language over the totally ordered alphabet $A$ (with $<$).  We say that $L$ is the \emph{numeration language}.
Genealogically (i.e., radix or length-lexicographic) ordering $L$ gives a one-to-one correspondence $\rep_S$ between $\N$ and $L$;
the \emph{$S$-representation} of $n$ is the $(n+1)$st word of $L$, and the inverse map, called the \emph{(e)valuation map}, is
denoted by $\val_S$.

\begin{example}\label{exa:little}
  Consider the ANS $S$ built on the language $\alpha^*\beta^*$ over the ordered alphabet $\{\alpha<\beta\}$. The first few words in the language are $\varepsilon,\alpha,\beta,\alpha\alpha,\ldots$. Hence, $\rep_S(3)=\alpha\alpha$ and $\val_S(\alpha\alpha)=3$.
\end{example}

In the following, we refer to the terminology introduced in \cite{peltomaki2021} (addable systems are called regular in \cite{Shallit2022logical}). It is convenient to introduce a new padding symbol~$\#$ which does not belong to the alphabet~$A$. We let $A_\#$ denote the set $A\cup\{\#\}$. We extend the evaluation map to $\#^*L$ by setting $\val_S(\#^n w)=\val_S(w)$ for all $w\in L$ and $n\in\mathbb{N}$.

\begin{definition}\label{def:addable}
  An abstract numeration system $S=(L,A,<)$ is {\em addable} if the following \emph{graph of addition}, denoted by $\mathcal{L}_+$, is regular:
\[
\left\{ \smatrix{ u \\ v \\ w} \in (\#^*L)^3 \cap (A_\# \times A_\# \times A_\#)^*
	\mid  \val_S(u) + \val_S(v) = \val_S(w) \right\} \setminus \smatrix{ \# \\ \# \\ \#} (A_\# \times A_\# \times A_\#)^*.
\]
\end{definition}

Notice that words in the numeration language $L$ do not start with $\#$; however, when dealing with tuples of such words, shorter $S$-representations are padded with leading $\#$'s to get words of equal length (so they can be processed by an automaton reading tuples of letters). Continuing \cref{exa:little}, for instance, the triplet $\smatrix{\#\alpha\\ \#\beta\\ \alpha\alpha}$ belongs to $\mathcal{L}_+$.

\begin{remark}
  Positional numeration systems (whose numeration language is regular) are special instances of ANS. Let us recall this classical setting. Let $U=(U_n)_{n\ge 0}$ be an increasing sequence of integers such that $U_0=1$. Any integer $n$ can be decomposed (not necessarily uniquely) as $n=\sum_{i=0}^t c_i\, U_i$ with non-negative integer coefficients $c_i$. The finite word $c_t\cdots c_0\in\mathbb{N}^*$ is a {\em $U$-representation} of $n$. If this representation is computed greedily \cite{Fraenkel,RigoBook}, then for all $j\le t$ we have $\sum_{i=0}^j c_i\, U_i<U_{j+1}$ and $\rep_U(n)=c_t\cdots c_0$ is said to be the {\em greedy} (or {\em normal}) $U$-representation of $n$.  By convention, the greedy representation of $0$ is the empty word~$\varepsilon$,  and the greedy representation of $n>0$ starts with a non-zero digit. An extra condition on the boundedness of $\sup_{i\ge 0} (U_{i+1}/U_i)$ implies that the digit-set for greedy representations is finite. For any $c_t\cdots c_0\in\mathbb{N}^*$, we let $\val_U(c_t\cdots c_0)$ denote the integer $\sum_{i=0}^t c_i\, U_i$. A sequence $U$ satisfying all the above conditions is said to define a {\em positional numeration system}. Any such system for which the numeration language $\rep_U(\mathbb{N})$ is regular is an ANS. For a positional numeration system, the existence of the digit $0$ permits to avoid the introduction of an extra symbol~$\#$. Padding can thus be achieved using leading zeroes.
\end{remark}

\begin{example}
  In this example, the numeration system has no digit~$0$ and has the property of being unambiguous. Consider the ANS~$S$ built on the language $L=\{1,2\}^*$ and the sequence $U=(2^n)_{n\ge 0}$. The first few words in $L$ are $\varepsilon,1,2,11,12,21,22,\ldots$. The $n$th word $d_k\cdots d_0$ in $L$ verifies $n=\sum_{i=0}^k d_i2^i$, but the greedy $U$-representation of $n$ is just its base-$2$ expansion over $\{0,1\}$ and is therefore not equal to $\rep_S(n)\in\{1,2\}^*$. The ANS $S$ is not, strictly speaking, a positional numeration system. Nevertheless the graph of addition for triplets of $S$-representations is regular. See \cref{fig:add12} where is depicted a DFA accepting the corresponding language reading least significant digit first, digits are processed from right to left. One simply has to deal with a carry $0,1,2$ stored within the state. Transitions are of the form
  $m\longrightarrow n$ with label $\smatrix{p\\ q\\ r}$,  for states $m,n\in\{0,1,2\}$ and letters $p,q,r\in\{\#,1,2\}$, if and only if $$m+p+q=r+2n$$
  where $\#$ is interpreted as $0$. 
  This is therefore an example of an addable ANS which is not a positional numeration system handling greedy expansions.
  \begin{figure}[h!tbp]
  \begin{center}
\begin{tikzpicture}[scale=.6,every node/.style={circle,minimum width=.5cm,inner sep=1pt}]
\node (0) at (0,0) [draw,circle] {\scriptsize{$0$}};
\node (1) at (10,0) [draw,circle] {\scriptsize{$1$}};
\node (2) at (20,0) [draw,circle] {\scriptsize{$2$}};
\draw [->] (0) edge[out=70,in=110,looseness=8] node[above=-.5] {\scriptsize{$\smatrix{1\\ 1\\ 2}$, $\smatrix{\#\\ i\\ i}$, $\smatrix{i\\ \#\\ i}$}} (0);
\draw [->] (1) edge[out=70,in=110,looseness=8] node[above=-.8] {\scriptsize{$\smatrix{1\\ 1\\ 1}$, $\smatrix{2\\ 1\\ 2}$ ,$\smatrix{1\\ 2\\ 2}$, $\smatrix{\#\\ 2\\ 1}$, $\smatrix{2\\ \#\\ 1}$}} (1);
\draw [->] (2) edge[out=70,in=110,looseness=8] node[above=-.5] {\scriptsize{$\smatrix{1\\ 2\\ 1}$, $\smatrix{2\\ 1\\ 1}$ ,$\smatrix{2\\ 2\\ 2}$}} (2);
\draw [->] (0) edge node[above=-.5,midway] {\scriptsize{$\smatrix{1\\ 2\\ 1}$, $\smatrix{2\\ 1\\ 1}$ ,$\smatrix{2\\ 2\\ 2}$}} (1);
 \draw [->] (1) edge node[above,midway] {\scriptsize{$\smatrix{2\\ 2\\ 1}$}} (2);
 \draw [->] (1) edge[out=225,in=-45] node[above=-.6] {\scriptsize{$\smatrix{\#\\ \#\\ 1}$, $\smatrix{1\\ \#\\ 2}$, $\smatrix{\#\\ 1\\ 2}$}} (0);
 \draw [->] (2) edge[out=225,in=-45] node[above=-.6] {\scriptsize{$\smatrix{i\\ \#\\ i}$, $\smatrix{\#\\ i\\ i}$ ,$\smatrix{1\\ 1\\ 2}$}} (1);
 \draw [->] (2) edge[out=225,in=-45] node[above] {\scriptsize{$\smatrix{\#\\ \#\\ 2}$}} (0);
 \draw [->] (-1,0) -- (0);
\end{tikzpicture}
\caption{A DFA accepting $\mathcal{L}_+$ reading words from left to right (where $i\in\{1,2\}$).}\label{fig:add12}
\end{center}
\end{figure}
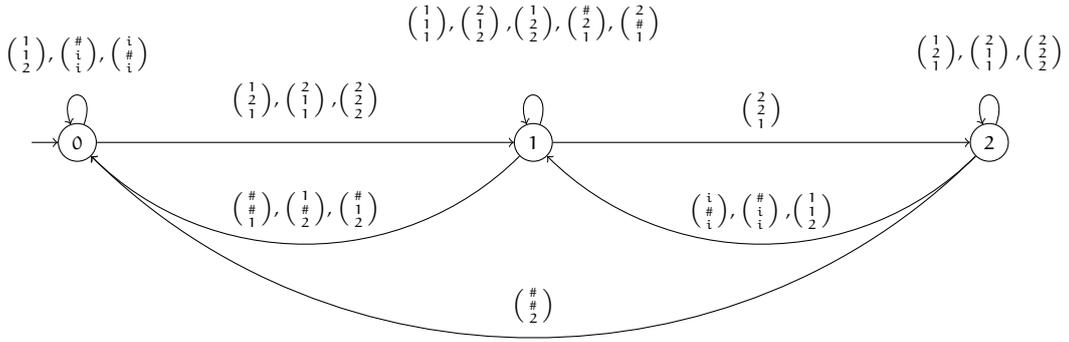
\end{example}

A \emph{deterministic finite automaton with output} (DFAO) $\mathcal{A}$ is a
DFA (with state set~$Q$) equipped with a mapping
$\tau\colon Q \to A$ (with $A$ an alphabet). The output $\mathcal{A}(w)$ of
$\mathcal{A}$ on a word $w$ is $\tau(q)$, where $q$
is the state reached by reading $w$ from the initial state.

\begin{definition}\label{def:Saut-kaut}
Let $S=(L,A,<)$ be an ANS.
  An infinite word $\mathbf{x}$ is {\em $S$-automatic} if there exists a DFAO $\mathcal{A}$ such that $\mathbf{x}[n]=\mathcal{A}(\rep_S(n))$.
  In particular, for an integer $k\ge 2$, if $\mathcal{A}$ is fed with the genealogically ordered language $L= \{\varepsilon\} \cup \{1,\ldots,k-1\}\{0,\ldots,k-1\}^*$, then $\mathbf{x}$ is said to be {\em $k$-automatic}.
If $\mathcal{A}$ is fed with the $U$-representations of integers, with $U$ a positional numeration system, $\infw{x}$ is said to be $U$-automatic.
\end{definition}


\begin{theorem}[\cite{Maes}]
\label{thm:morphicIsAutomatic}
A word $\infw{x}$ is morphic if and only if it is $S$-automatic
for some abstract numeration system $S$.
\end{theorem}

We note that the proof of the above theorem shows that the equivalence is completely effective: given the morphisms producing the word $\infw{x}$, one can construct an ANS $S$ and a DFAO generating~$\infw{x}$, and vice versa.

Fix $s \in A^*$. For a word $\mathbf{x}$, define the subsequence $\mathbf{x} \circ s$ by
$(\mathbf{x}\circ s)[n]:=
\mathbf{x}[\val_S(p_{s,n}\, s)]$,
where $p_{s,n}$ is the $n$th word in the genealogically ordered language $Ls^{-1}=\{u\in A^*\mid us\in L\}$. The \emph{$S$-kernel} of the word $\mathbf{x}$ is defined as the set of words $\{\mathbf{x}\circ s \mid s\in A^*\}$. The following theorem is critical to our arguments. Details are given in \cite[Prop.~3.4.12--16]{cant2010}.
\begin{theorem}[\cite{Maes}]\label{thm:Ukernel}
A word $\mathbf{x}$ is $S$-automatic if and only if its $S$-kernel is finite.
\end{theorem}

Again, the theorem is completely effective: with the underlying ANS $S$ fixed, given (a Turing machine generating) $\infw{x}$, and (the cardinality of) the $S$-kernel, one can compute the DFAO generating~$\infw{x}$, and vice versa.

\begin{example}
Consider the Fibonacci numeration system based on the sequence of Fibonacci numbers $(F_n)_{n\ge 0}$ with $F_0 = 1$, $F_1 = 2$,
and $F_{n+1} = F_n + F_{n-1}$ for $n\geq 1$.
The first few terms of the associated subsequences $\mu_s:\mathbb{N}\to\mathbb{N}$, such that $(\mathbf{x}\circ s)[n]=\mathbf{x}[\mu_s(n)]$, are given in \cref{tab:fib-ker}.
One simply computes the numerical value of all the Fibonacci representations with the suffix $s$.
\end{example}

   \begin{table}[h!t]
\centering
\begin{tabular}{r|lr|l}
    $s$ & $(\mu_s(n))_{n\ge 0}$ & $s$ & $(\mu_s(n))_{n\ge 0}$ \\
    \hline
$\varepsilon$ & 0, 1, 2, 3, 4, 5, 6, 7, 8, 9, \ldots & 01& 4, 6, 9, 12, 14, 17, 19, 22, 25, \ldots \\
0& 2, 3, 5, 7, 8, 10, 11, 13, 15, \ldots &  00& 3, 5, 8, 11, 13, 16, 18, 21, 24, \ldots\\
1 & 1, 4, 6, 9, 12, 14, 17, 19, 22, 25, \ldots & 10& 2, 7, 10, 15, 20, 23, 28, 31, 36, 41, \ldots  \\
   \end{tabular}
    \caption{The first few terms of some subsequences $\mu_s$
    for the Fibonacci numeration system.}
    \label{tab:fib-ker}
  \end{table}

Notice that some kernel elements $\mathbf{x}\circ s$ may be finite; more precisely, this occurs exactly when the language $L s^{-1}$ is finite.
Our reasoning will not be affected by such particular cases, and we let the reader adapt it to such situations.


In \cref{sec:addable}, we require $S$ to be addable. Note that these assumptions of having a numeration language that is regular and addable are shared by many classical systems. For instance,  the usual integer base numeration systems or the Fibonacci numeration system have all the assumed properties. For the latter system, the minimal automaton (reading most significant digits first) of $\mathcal{L}_+$ has $17$ states (its transition table is given in \cite{MousaviShallit}). The one reading least significant digits first has $22$ states. The largest known family of positional systems with all these properties (addable with $\rep_U(\N)$ being regular) is the one of those based on a linear recurrence sequence whose characteristic polynomial is the minimal polynomial of a Pisot number \cite{BruyereHansel,RigoBook}. One practical difficulty when one wants to use automatic provers (such as {\tt Walnut}
\cite{Walnut}) is to be
able to provide the relevant automaton for addition.

\subsection{Link with first-order logic}\label{ss:logic}



The result stated below is at the origin of {\tt Walnut}. It relies on the effective transformation of formulae to automata.
It was first stated for integer-based systems. Making use of \cite[Lem.~37 and Thm.~55]{CharlierCS2022regular}, it was extended to addable systems:

\begin{theorem}[{\cite[Thm. 6.4.1]{Shallit2022logical}}]\label{thm:walnut}
Let $S$ be an addable numeration system. There is an algorithm that, given a formula $\varphi$ with no free
variables, phrased in first-order logic, using only the universal and existential
quantifiers, addition and subtraction of variables and constants, logical operations, comparisons, and indexing into a given $S$-automatic sequence $\mathbf{x}$, will decide the truth of the formula $\varphi$.
Furthermore, if $\varphi$ has $t\ge 1$ free variables, the algorithm produces a DFA that recognizes the language of all representations of $t$-tuples of natural numbers that make $\varphi$ evaluate to true.
\end{theorem}

As already mentioned in \cite{Guo}, the boundary sequence of a $k$-automatic
word~$\mathbf{x}$ may be defined by means of a first-order formula and
therefore automaticity readily follows. This extends to addable systems~$S$: let $\mathbf{x}\in B^\N$ be $S$-automatic
for an addable system $S$. The above theorem implies that, for all $b\in B$, 
we have a formula $\varphi_b(n)$ which is true if and only if
$\mathbf{x}[n]=b$. We have
$(u_1\cdots u_\ell,v_1\cdots v_\ell)\in\partial_{\mathbf{x},\ell}[m]$ if and
only if
\[
(\exists i)\bigwedge_{j=1}^{\ell}\varphi_{u_j}(i+j-1)\wedge \bigwedge_{j=1}^{\ell}\varphi_{v_j}(i+m + j-1).
\]
For each subset $R$ of $A^{\ell} \times A^{\ell}$ there is thus a formula $\psi_R(m)$ which is true if and only if $\partial_{\mathbf{x},\ell}[m] = R$.
We may now apply \cref{thm:walnut} to conclude that
$\partial_{\mathbf{x},\ell}$ is $S$-automatic. These arguments appear in \cite[\S 8.1.11]{Shallit2022logical} in the case $\ell = 1$.

\begin{example}\label{exa:walnut1}
  We use the strategy described in \cite[Sec.~8.1.11]{Shallit2022logical},
  where the boundary sequence of the Fibonacci word is computed. Here, we consider the
  Thue--Morse word~$\mathbf{t}$ and show how to get its $2$-boundary sequence
$$\partial_{\mathbf{t},2} =  abcdedfdgdgdgdcdgdgdgdgdgdgdgdfdgd\cdots$$
  using {\tt Walnut}. Let $\alpha,\beta,\gamma,\delta\in\{0,1\}$. If there exists some position $i$ such that $(\mathbf{t}[i]\mathbf{t}[i+1],\mathbf{t}[i+n]\mathbf{t}[i+n+1])=(\alpha\beta,\gamma\delta)$, then this pair belongs to $\partial_{\mathbf{t},2}[n]$. In {\tt Walnut}, depending on the value of $\alpha,\beta,\gamma,\delta$, we provide sixteen definitions of the form
\begin{center}
\verb!def TMbound!$\alpha\beta\gamma\delta$ \verb!"Ei T[i]=@!$\alpha$\verb! & T[i+1]=@!$\beta$\verb! & T[i+n]=@!$\gamma$\verb! & T[i+n+1]=@!$\delta$\verb!";!
\end{center}
which create deterministic automata recognizing base-$2$ expansions of the sets $$T_{\alpha\beta,\gamma\delta}=\{n\in\mathbb{N}:(\alpha\beta,\gamma\delta)\in\partial_{\mathbf{t},2}[n]\}.$$ In particular, \verb!$TMbound!$\alpha\beta\gamma\delta$\verb!(n)! evaluates to {\tt TRUE} whenever $n$ belongs to $T_{\alpha\beta,\gamma\delta}$. A direct inspection shows that only seven different $2$-boundary sets occur in $\partial_{\mathbf{t},2}$:
\begin{align*}
a &:= \{0,1\}^2\times \{0,1\}^2 \setminus \{(00,00),(00,01),(01,11), (10,00), (11,10),(11,11) \}, \\
b &:= \{0,1\}^2\times \{0,1\}^2 \setminus \{(00,00), (00,11), (01,10), (10,01), (11,00), (11,11)\},\\
c &:=\{0,1\}^2 \times \{0,1\}^2 \ \setminus \ \{ (00,10), (01,00), (10,11), (11,01) \},\\
d &:= \{0,1\}^2 \times \{0,1\}^2 \ \setminus \ \{(00,00), (00,11), (11,00), (11,11)\},\\
e &:= \{0,1\}^2 \times \{0,1\}^2 \ \setminus \ \{ (00,11), (11,00) \},\\
f &:= \{0,1\}^2 \times \{0,1\}^2 \ \setminus \ \{(00,01), (01,11), (10,00), (11,10) \},\\
g &:= \{0,1\}^2 \times \{0,1\}^2.
\end{align*}
This can be checked as follows. For the set $a$, we provide the following definition
\begin{verbatim}
def TMbounda "~$TMbound0000(n) & ~$TMbound0001(n) & $TMbound0010(n) & 
$TMbound0011(n) & $TMbound0100(n) & $TMbound0101(n) & $TMbound0110(n) & 
~$TMbound0111(n) & ~$TMbound1000(n) & $TMbound1001(n) & $TMbound1010(n) & 
$TMbound1011(n) & $TMbound1100(n) & $TMbound1101(n) & ~$TMbound1110(n) & 
~$TMbound1111(n)";
\end{verbatim}
where \verb!$TMbound0(n)! evaluates to {\tt TRUE} whenever $\partial_{\mathbf{t},2}[n]=a$. Similar definitions are readily written for $b,\ldots,g$. The following expression evaluates to {\tt TRUE}
\begin{verbatim}
eval TM2boundarycheck "An (n>1) => (($TMbounda(n) | $TMboundb(n) | $TMboundc(n) | 
$TMboundd(n) | $TMbounde(n) | $TMboundf(n) | $TMboundg(n)))":
\end{verbatim}
meaning that we are not missing any $2$-boundary set. 
We combine the seven automata produced by {\tt Walnut} accepting base-$2$ expansion of the integers $n$ such that $\partial_{\mathbf{t},2}[n]$ is a particular letter in $\{a,\ldots,g\}$ into the DFAO depicted in \cref{fig:dfaot2} using the command
\begin{verbatim}
combine TM2boundarySequence TMbounda TMboundb TMboundc
                            TMboundd TMbounde TMboundf TMboundg:
\end{verbatim}

\begin{figure}[h!tbp]
  \begin{center}
\begin{tikzpicture}[scale=.6,every node/.style={circle,minimum width=.5cm,inner sep=1pt}]

\node (0) at (0,0) [draw,circle] {\scriptsize{$\#$}};
 \node (1) at (2,0) [draw,circle] {\scriptsize{$\$ $}};
 \node (2) at (4,-2) [draw,circle] {\scriptsize{$a$}};
 \node (3) at (4,2) [draw,circle] {\scriptsize{$b$}};
 \node (4) at (6,-3) [draw,circle] {\scriptsize{$c$}};
 \node (5) at (6,0) [draw,circle] {\scriptsize{$d$}};
 \node (6) at (6,3) [draw,circle] {\scriptsize{$e$}};
 \node (7) at (8,-2) [draw,circle] {\scriptsize{$f$}};
 \node (8) at (8,2) [draw,circle] {\scriptsize{$g$}};
 \draw [->] (0) edge[out=70,in=110,looseness=8] node[above] {\scriptsize{$0$}} (0);
 \draw [->] (0) edge node[above,midway] {\scriptsize{$1$}} (1);

 \draw [->] (1) edge node[below] {\scriptsize{$0$}} (2);

 \draw [->] (1) edge node[above] {\scriptsize{$1$}} (3);

 \draw [->] (2) edge node[above] {\scriptsize{$0$}} (4);

 \draw [->] (2) edge node[above] {\scriptsize{$1$}} (5);

 \draw [->] (3) edge node[above] {\scriptsize{$1$}} (5);

 \draw [->] (3) edge node[above] {\scriptsize{$0$}} (6);

 \draw [->] (4) edge node[right] {\scriptsize{$1$}} (5);

 \draw [->] (4) edge[out=0,in=230] node[below] {\scriptsize{$0$}} (7);

 \draw [->] (5) edge[out=340,in=20,looseness=7] node[right] {\scriptsize{$1$}} (5);

 \draw [->] (5) edge[out=60,in=200] node[above] {\scriptsize{$0$}} (8);

 \draw [->] (6) edge node[left] {\scriptsize{$1$}} (5);

 \draw [->] (6) edge node[above] {\scriptsize{$0$}} (8);

 \draw [->] (7) edge node[above] {\scriptsize{$0$}} (4);

 \draw [->] (7) edge node[right] {\scriptsize{$1$}} (5);

 \draw [->] (8) edge[out=240,in=40] node[below] {\scriptsize{$1$}} (5);

 \draw [->] (8) edge[out=345,in=15,looseness=7] node[right] {\scriptsize{$0$}} (8);

 \draw [->] (-1,0) -- (0);
\end{tikzpicture}
\caption{A DFAO producing $\partial_{\mathbf{t},2}$.}\label{fig:dfaot2}
\end{center}
\end{figure}
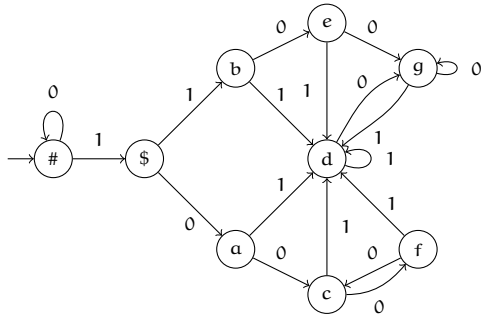
Inspecting \cref{fig:dfaot2}, we notice that $a$, $b$ and $e$ appear exactly once. On the other hand, we have $\partial_{\mathbf{t},2}[n] = c$ precisely when $n$
is an even power of $2$, while $\partial_{\mathbf{t},2}[n]= f$ if and only if $n$ is an odd
power of $2$ strictly larger than $2$. Finally $\partial_{\mathbf{t},2}[n] = d$ precisely when $n$ is odd and at least $5$.

Through the effective conversion of the automaton to a morphic representation of the word, we get the $2$-uniform morphism generating the $2$-boundary sequence (prepended with two symbols $\# \$ $, because the $2$-boundary sequence is indexed starting at $2$):
\[
\begin{cases}
\# 			& \mapsto \# \$ \\
\$ 			& \mapsto ab \\
a 	& \mapsto cd
\end{cases}
\qquad\qquad
\begin{cases}
b 	& \mapsto ed \\
c 	& \mapsto fd \\
d 	& \mapsto gd 
\end{cases}
\qquad\qquad
\begin{cases}
e 	& \mapsto gd \\
f 	& \mapsto cd \\
g 	& \mapsto gd.
\end{cases}
\]
\end{example}

\begin{example}\label{exa:walnut2}
  For the Fibonacci word~$\mathbf{f}$, the strategy is similar to the one given in the previous example. Instead of binary expansions, we simply make use of Fibonacci expansions which are also available in {\tt Walnut}. For instance
\begin{verbatim}
def Fbound00 "?msd_fib Ei F[i]=@0 & F[i+n]=@0";
\end{verbatim}
  is such that \verb!Fbound00(n)! evaluates to {\tt TRUE} whenever $(0,0)$ belongs to $\partial_{\mathbf{f}}[n]$. The details and the resulting automaton can be found in \cite[Sec.~8.1.11]{Shallit2022logical}. 
  Doing a similar job for the $2$-boundary sequence, here at most $9$ (and not sixteen, as in the previous example) boundary pairs may occur because $\mathbf{f}$ does not contain $11$ as a factor in $\mathbf{f}$. We can check that $\partial_{\mathbf{f},2}$ is made of five different boundary sets. The resulting DFAO is depicted in \cref{fig:fib2bound} (the sink state reached when reading a factor $11$ is not represented).
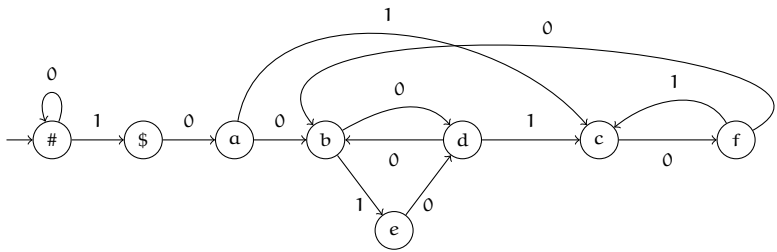
\begin{figure}[h!tbp]
  \begin{center}
\begin{tikzpicture}[scale=.6,every node/.style={circle,minimum width=.5cm,inner sep=1pt}]

\node (0) at (0,0) [draw,circle] {\scriptsize{$\#$}};
 \node (1) at (2,0) [draw,circle] {\scriptsize{$\$ $}}; 
 \node (2) at (4,0) [draw,circle] {\scriptsize{$a$}};
 \node (3) at (6,0) [draw,circle] {\scriptsize{$b$}};
 \node (4) at (12,0) [draw,circle] {\scriptsize{$c$}};
 \node (5) at (9,0) [draw,circle] {\scriptsize{$d$}};
 \node (6) at (7.5,-2) [draw,circle] {\scriptsize{$e$}};
 \node (7) at (15,0) [draw,circle] {\scriptsize{$f$}};
 \draw [->] (0) edge[out=70,in=110,looseness=8] node[above] {\scriptsize{$0$}} (0);
 \draw [->] (0) edge node[above,midway] {\scriptsize{$1$}} (1);
 \draw [->] (1) edge node[above] {\scriptsize{$0$}} (2);
 \draw [->] (2) edge node[above] {\scriptsize{$0$}} (3);
 \draw [->] (2) edge[out=80,in=-230] node[above] {\scriptsize{$1$}} (4);
 \draw [->] (3) edge[out=30,in=-230] node[above] {\scriptsize{$0$}} (5);
 \draw [->] (3) edge node[below] {\scriptsize{$1$}} (6);
 \draw [->] (6) edge node[below] {\scriptsize{$0$}} (5);
 \draw [->] (5) edge node[below] {\scriptsize{$0$}} (3);
 \draw [->] (5) edge node[above] {\scriptsize{$1$}} (4);
 \draw [->] (4) edge node[below] {\scriptsize{$0$}} (7);
 \draw [->] (7) edge[out=120,in=40] node[above] {\scriptsize{$1$}} (4);
  \draw [->] (7) edge[out=30,in=-230] node[above] {\scriptsize{$0$}} (3);
 \draw [->] (-1,0) -- (0);
\end{tikzpicture}
 \caption{A DFAO producing the $2-$boundary sequence of $\mathbf{f}$.}
 \label{fig:fib2bound}
\end{center}
\end{figure}
We thus get the morphism generating the $2$-boundary sequence of~$\mathbf{f}$ prepended with two symbols $\# \$ $:
\[
\begin{cases}
\# 			& \mapsto \# \$ \\
\$ 			& \mapsto a \\
a 	& \mapsto bc
\end{cases}
\qquad\qquad
\begin{cases}
  b& \mapsto de\\
  c& \mapsto f\\
  \end{cases}
\qquad\qquad
\begin{cases}
  d& \mapsto bc\\
  e& \mapsto d\\
  f& \mapsto bc.\\
\end{cases}
\]

Finally, we also computed the DFAO for the boundary sequence of the Tribonacci word (reading Tribonacci expansions),
the fixed point of the morphism $0 \mapsto 01$, $1\mapsto 02$, $2\mapsto 0$. Surprisingly, the (minimal) automaton produced
by {\tt Walnut} has 118 states. It also reveals that there are exactly seven boundary sets appearing
in the boundary sequence of the Tribonacci word. One can show, using {\tt Walnut}, that the first and second sets
in the boundary sequence appear exactly once, otherwise the boundary sets appear infinitely often.
\end{example}

\section{On the boundary sequences of automatic words}\label{sec:automatic}
In this section we provide the first of our main contributions, an alternative proof (not relying on \cref{thm:walnut}) to the fact that
an $S$-automatic word has an $S$-automatic boundary sequence whenever $S$ is addable.
We then show that this result does not necessarily hold for a non-addable system.

\subsection{Addable systems: automatic boundary sequences}\label{sec:addable}

For the sake of presentation, we only consider the case of the $1$-boundary sequence.
Our proof provides a precise description of a set containing the
$S$-kernel of $\partial_\mathbf{x}$ in terms of three equivalence
relations based on the kernel of~$\mathbf{x}$, the graph of addition,
and the numeration language; see~\eqref{eq:new-equiv-rel}. This set is finite, and so \cref{thm:morphicIsAutomatic} gives the claim.
In particular, one is the Myhill--Nerode congruence associated with the graph of addition since we have to consider the elements $\mathbf{x}[i]$ and $\mathbf{x}[i+m]$ for some $m>0$.
For $\ell>1$, the only technical difference is that we have to consider longer factors $\mathbf{x}[i] \cdots \mathbf{x}[i+\ell-1]$ and $\mathbf{x}[i+m] \cdots \mathbf{x}[i+m+\ell-1]$.

\begin{theorem}\label{the:kernel}
Let $S=(L,A,<)$ be an addable ANS and let $\mathbf{x}$ be an $S$-automatic word. The boundary sequence $\partial_\mathbf{x}$ is $S$-automatic.
\end{theorem}

\begin{proof}
 Thanks to \cref{thm:Ukernel}, the $S$-kernel of $\mathbf{x}$ is finite, say of cardinality $m$.
 Moreover, since $L$ and $\mathcal{L}_+$ are regular, the following two sets of languages are finite by the Myhill--Nerode theorem \cite[Sec.~3.9]{ShallitAutomata}, say of cardinality
 $k$ and $\ell$, respectively:
\[
\{Ls^{-1}\mid s\in A^*\} \quad \text{and} \quad
\Big\{\mathcal{L}_+ \smatrix{s \\ t \\ r}^{-1} \mathrel{\Big|} \smatrix{s \\ t \\ r} \in (A \times A_\# \times A_\#)^* \Big\}.
\]
Let $\partial_\mathbf{x}$ be the boundary sequence of $\mathbf{x}$. An element of the $S$-kernel of $\partial_\mathbf{x}$ is given by
$
\partial_\mathbf{x}\circ s = 
\partial_\mathbf{x}[\val_S(p_{s,0}s)]\ 
\partial_\mathbf{x}[\val_S(p_{s,1}s)]\ 
\partial_\mathbf{x}[\val_S(p_{s,2}s)]\cdots
$
where $p_{s,n}$ is the $n$th word in the language $Ls^{-1}$, $n\ge 0$.
Let us inspect the $n$th term of such an element of the kernel: it is precisely the set
\begin{equation}
\label{eq:kernelElement}
\partial_\mathbf{x}[\val_S(p_{s,n}s)]=\{(\mathbf{x}[i],\, \mathbf{x}[i+\val_S(p_{s,n}s)]) \mid i\ge 0\}
\end{equation}
of pairs of letters.
Let $t$, $r$ be length-$|s|$ suffixes of words in $\#^*L$ for which
$\mathcal{L}_+(s,t,r)^{-1}$ is non-empty. There exist words $w$, $x$, $y$ such
that $ws$, $xt$, $yr \in \#^*L$ and $\val_S(ws)+\val_S(xt)=\val_S(yr)$. We
let $\mathcal{P}(s)$ denote the set of such pairs $(t,r)\in (A_\#\times A)^{|s|}$.
Now partition \eqref{eq:kernelElement} depending on the suffixes of length $|s|$ of
$\rep_S(i)$ and $\rep_S(i+\val_S(p_{s,n}s))$: we may write
\[
\partial_\mathbf{x}[\val_S(p_{s,n}s)] = \bigcup_{(t,r)\in \mathcal{P}(s)}\left\{
(\mathbf{x}[\val_S(xt)],\, \mathbf{x}[\val_S(yr)])
\mathrel{\Big|} \smatrix{w \\ x \\ y} \in \mathcal{L}_+\smatrix{s \\ t \\ r}^{-1} \wedge w\in \#^*p_{s,n}
\right\}.
\]
Roughly speaking, we look at all pairs of positions such that the first one is represented by a word ending with $t$, the second position is a shift of the first one by $\val_S(p_{s,n}s)$ and is represented by a word ending with $r$.

For convenience, we set
$
L(s,t,r,n):=\mathcal{L}_+ \smatrix{s \\ t \\ r}^{-1} \cap \left( \#^*p_{s,n} \times A_\#^* \times A_\#^* \right)$ for all $n\ge 0$.
Note that if $\mathcal{L}_+(s,t,r)^{-1}= \mathcal{L}_+(s',t',r')^{-1}$ and $Ls^{-1}=Ls'^{-1}$ then, for all~$n$, $L(s,t,r,n)=L(s',t',r',n)$.  Indeed, the second condition means that $p_{s,n}=p_{s',n}$ for all $n$.

\medskip

\noindent
$\bullet$ Ordering $L(s,t,r,n)$.

\noindent
For each $w$, $x$ of the same length, there is at most one $y$ not starting with $\#$ such that $(w,x,y)$ belongs to $\mathcal{L}_+(s,t,r)^{-1}$. Similarly if $y$ does not start with $\#$, for each $w\in A_\#^{|y|}$ (resp., $x\in A_\#^{|y|}$) there is at most one $x$ (resp., $w$) such that $(w,x,y)$ belongs to $\mathcal{L}_+(s,t,r)^{-1}$.

Now let $(w,x,y)$ and $(w',x',y')$ in $L(s,t,r,n)$. We will always assume (this is not a restriction) that triplets do not start with $(\#,\#,\#)$ --- otherwise, different triplets may have the same numerical value. Note that $\val_S(x)<\val_S(x')$ if and only if $\val_S(y)<\val_S(y')$.
Indeed, $w,w'$ both belong to $\#^*p_{s,n}$, thus $\val_S(ws)=\val_S(p_{s,n} s)=\val_S(w's)$. We have
\[
\val_S(yr)-\val_S(xt)=\val_S(ws)=\val_S(w's)=\val_S(y'r)-\val_S(x't),
\]
so,
\[
\val_S(yr)-\val_S(y'r)=\val_S(xt)-\val_S(x't).
\]
Since $S$ is an ANS, if $\val_S(xt)>\val_S(x't)$, this means that, discarding the possible leading $\#$'s because they have no effect on the evaluation, $xt$ occurs after $x't$ in the genealogically ordered language. So $x$ has to be genealogically larger than $x'$.
Again discarding the possible leading $\#$'s, the above equality means that $x$ is genealogically less than $x'$ if and only if the same holds for $y$ and $y'$. We can thus order $L(s,t,r,n)$ by listing in increasing genealogical order the second component of the elements, and therefore the $j$th element of $L(s,t,r,n)$ is well-defined.

\medskip

\noindent
$\bullet$ Defining two subsequences by the maps $\lambda_{s,t,r,n}:\mathbb{N}\to\mathbb{N}$ and $\mu_{s,t,r,n}:\mathbb{N}\to\mathbb{N}$.

\noindent
Let $(w_j,x_j,y_j)$ be the $j$th element in $L(s,t,r,n)$ with $j\ge 0$. After removing the leading $\#$'s, the word 
$x_j$ belongs to $Lt^{-1}\cup\{\varepsilon\}$, which can also be genealogically ordered. We let $\lambda_{s,t,r,n}(j)$ denote the index (i.e., position counting from $0$) of $x_j$ within this language. Similarly, the word 
$y_j$ belongs to $Lr^{-1}\cup\{\varepsilon\}$ and has an index $\mu_{s,t,r,n}(j)$ within this language.

Note that if $\mathcal{L}_+(s,t,r)^{-1}= \mathcal{L}_+(s',t',r')^{-1}$,
$Ls^{-1}=Ls'^{-1}$, and $Lt^{-1}=Lt'^{-1}$ then, for all~$n$, the maps
$\lambda_{s,t,r,n}$ and $\lambda_{s',t',r',n}$ are the same. Indeed, the
first two conditions imply that $L(s,t,r,n)=L(s',t',r',n)$. Similarly, if
$\mathcal{L}_+(s,t,r)^{-1}= \mathcal{L}_+(s',t',r')^{-1}$,
$Ls^{-1}=Ls'^{-1}$, and $Lr^{-1}=Lr'^{-1}$ then, for all~$n$, the maps
$\mu_{s,t,r,n}$ and $\mu_{s',t',r',n}$ are the same.

We now obtain
\begin{align*}
  \partial_\mathbf{x}[\val_S(p_{s,n}s)]
  &=\!
     \bigcup_{(t,r)\in \mathcal{P}(s)} \!\! \left\{ 
( \mathbf{x}[\val_S(xt)],\, \mathbf{x}[\val_S(yr)] )
\mathrel{\Big|} \smatrix{w \\ x \\ y}\in \mathcal{L}_+ \smatrix{s \\ t \\ r}^{-1} \wedge w\in \#^*p_{s,n}
                                     \right\}\\
  &= \! \bigcup_{(t,r)\in \mathcal{P}(s)} \!\! \left\{
( \mathbf{x}[\val_S(x_jt)],\, \mathbf{x}[\val_S(y_jr)] )
\mathrel{\Big|} \smatrix{w_j \\ x_j \\ y_j} \in L(s,t,r,n), j\ge 0
      \right\}\\
  &= \! \bigcup_{(t,r)\in \mathcal{P}(s)} \!\! \left\{
( (\mathbf{x}\circ t)[\lambda_{s,t,r,n}(j)], (\mathbf{x}\circ r)[\mu_{s,t,r,n}(j)] )
\mid j\ge 0 \right\}.
\end{align*}

Let us define an equivalence relation $\sim$ on triplets by $(s,t,r)\sim(s',t',r')$ if and only if all the following hold: 
\begin{align}\label{eq:new-equiv-rel}
 \mathcal{L}_+\smatrix{ s \\ t \\ r}^{\!\!\!-1} 
= \mathcal{L}_+ \smatrix{s' \\ t' \\ r'}^{\!\!\!-1}\!\!\!,\quad
& Ls^{-1} = Ls'^{-1}, \quad
 Lt^{-1} = Lt'^{-1}, \quad
 Lr^{-1}=Lr'^{-1}, \\
&\mathbf{x} \circ t = \mathbf{x} \circ t', \quad \text{ and } \quad
\mathbf{x}\circ r = \mathbf{x}\circ r'. \nonumber
\end{align}

Since we have regular languages and the kernel of $\mathbf{x}$ is finite by assumption, this relation has a finite index (bounded by $\ell k^3 m^2$). Given $s$, the set $\{(s,t,r)\mid (t,r)\in\mathcal{P}(s)\}$ can be replaced by a set $\Lambda(s)$ of representatives of the equivalence classes for $\sim$. Since $\sim$ has a finite index, there are finitely many possible subsets of the form $\Lambda(s)$. So, we can write
\[
\partial_\mathbf{x}[\val_S(p_{s,n}s)] =
 \bigcup_{(b,c,a)\in \Lambda(s)}
 	\left\{
((\mathbf{x}\circ c)[\lambda_{b,c,a,n}(j)], (\mathbf{x}\circ a)[\mu_{b,c,a,n}(j)])
\mid j\ge 0 \right\}.
\]
                                     Now if $s$ and $s'$ are such that $Ls^{-1}=Ls'^{-1}$ and $\Lambda(s)=\Lambda(s')$, then $\partial_\mathbf{x}\circ s = \partial_\mathbf{x}\circ s'$. This proves that the kernel of $\partial_\mathbf{x}$ is finite (of size bounded by $k\cdot 2^{\ell k^3 m^2}$).
\end{proof} 

\subsection{A family of non-addable systems}\label{ss:non-add}

In this section,  we show that the addability assumption on the numeration system is not necessary for the boundary sequence of an automatic word to be itself automatic. With \cref{exa:astarbstar,exa:astarbstar2}, we consider $S$-automatic sequences based on a non-addable ANS $S$ but such that the corresponding boundary sequences are still $S$-automatic. The first lemma is merely an observation that we will frequently use.

\begin{lemma}\label{lem:diff}
  Let $\mathbf{w}\in\{0,1\}^\mathbb{N}$ be an aperiodic binary word having arbitrarily long blocks of $0$s. Its boundary sequence $\partial_{\mathbf{w}}$ is over the alphabet $\{a,b\}$ where $a:=\{(0,0),(0,1),(1,0)\}$ and $b:=\{0,1\} \times \{0,1\}$. We have $\partial_{\mathbf{w}}[k]=b$ if and only if there exists $m>n\ge 0$ such $\mathbf{w}[m]=\mathbf{w}[n]=1$ and $k=m-n$.
\end{lemma}

\begin{proof}
  By assumption, $\mathbf{w}$ contains factors of the form $0^{k+1}$, $0^k1$ and $10^k$ for all $k\ge 1$. So a boundary set can either be $a$ or $b$.
  In $\mathbf{w}$, a window of length~$k$ will start with $1$ and is followed by a $1$ only if there exists $n$ such that $\mathbf{w}[n]=1$ and $\mathbf{w}[n+k]=1$.
\end{proof}

Let $s$ be an integer. In the next three examples, we consider morphic words~$\mathbf{w}_s$ from the same family. They are the image under the same coding (up to erasing the first symbol) of a fixed point of $g_s \colon 0\mapsto 01, 1\mapsto 12^s, 2\mapsto 2$, for $s \geq 1$. We show that the corresponding boundary sequences $\partial_{\mathbf{w}_s}$ may exhibit quite different behaviors: for $s=1$, it is constant; for $s=2$, it is periodic of period~$4$, and for $s\ge 3$, it is aperiodic.

\begin{example}\label{exa:astarbstar}
Consider the morphisms $g_1\colon 0\mapsto 01, 1\mapsto 12, 2\mapsto 2$ and $f\colon 0\mapsto \varepsilon, 1\mapsto 1, 2\mapsto 0$, and the word
\[
\mathbf{w}_1
= f(g_1^\omega(0))
= 1 1 0 1 0 0 1 0 0 0 1 0 0 0 0 1 0 0 0 0 \cdots.
\]
It is the characteristic sequence of triangular numbers (\cite[A000217]{OEIS}).
A \emph{triangular number} is any integer of the form $T_n:=\binom{n+1}{2} = \frac{n(n+1)}{2}$ for $n\ge 0$.
The sequence $(T_n)_{n\ge 0}$ starts with $0,1,3,6,10,15,21,28,36,45,55$.
 \end{example}

The ANS $S = (\alpha^*\beta^*,\{\alpha,\beta\},\alpha<\beta)$ is known to be non-addable, \cite[Thm.~17]{LecomteRigo2001}. The reason is that multiplication by a constant generally does not preserve $S$-recognizability, hence addition cannot have this property.

 \begin{proposition}\label{pro:add1}
 Let $S = (\alpha^*\beta^*,\{\alpha,\beta\},\alpha<\beta)$.
The characteristic sequence $\mathbf{w}_1$ of the set of triangular numbers given in \cref{exa:astarbstar} is $S$-automatic.
The boundary sequence $\partial_{\mathbf{w}_1}$ is $S$-automatic.  In particular, it is constant.
\end{proposition}


\begin{proof} 
  Let $s$ be a suffix of a word in $\alpha^*\beta^*$. We make use of the same notation as in the proof of \cref{the:kernel}. Let $p_{s,n}$ be the $n$th word in $\alpha^*\beta^*s^{-1}$,  for $n\ge 0$.
  As in \cref{eq:kernelElement}, the $n$th term of an element of the $S$-kernel of $\partial_{\mathbf{w}_1}$ is given by
  \begin{equation}
    \label{eq:formk}
      (\partial_{\mathbf{w}_1}\circ s)[n]=\{\mathbf{w}_1[i]\mathbf{w}_1[i+\val_S(p_{s,n}s)]\mid i\ge 0\}.
  \end{equation}
  For the numeration language of interest, the admissible suffixes $s$ are of the form $\alpha^\ell \beta^k$ for some $\ell,k\ge 0$. If $\ell>0$, then $p_{s,n}=\alpha^n$ and
  \begin{equation}
    \label{eq:proof11}
    (\partial_{\mathbf{w}_1}\circ \alpha^\ell \beta^k)[n]=\left\{\mathbf{w}_1[i]\mathbf{w}_1[i+T_{n+\ell+k}+k]\mid i\ge 0\right\}
  \end{equation}
  because
  \[
\val_S (\alpha^i \beta^j) = \frac{1}{2} (i+j) (i+j+1) + j = T_{i+j} + j
\]
for all $i,j\ge 0$ (see~\cite[Ex. 2.18]{RigoBook}). By \cref{lem:diff}, $(\partial_{\mathbf{w}_1}\circ \alpha^\ell \beta^k)[n]$ always contains $(0,0)$, $(0,1)$, $(1,0)$.
For all $n\ge 0$, there exists $j$ such that $T_{n+\ell+k}+k=T_{j+1}-T_j=j+1$. Taking $i=T_j$ in \cref{eq:proof11} shows that $(1,1)$ also belongs to $(\partial_{\mathbf{w}_1}\circ \alpha^\ell \beta^k)[n]$ which is thus equal to the set $b$. So the sequence $\partial_{\mathbf{w}_1}\circ \alpha^\ell \beta^k$ is constant.

Now, consider a suffix $s$ of the form $\beta^k$, then $p_{s,n}=\alpha^i\beta^j$ for which $\val_S(p_{s,n})=n$. We have, for $i,j\ge 0$, 
$$(\partial_{\mathbf{w}_1}\circ \beta^k)[\val_S(\alpha^i\beta^j)]=\left\{\mathbf{w}_1[t]\mathbf{w}_1[t+\val_S(\alpha^i\beta^{j+k})]\mid t\ge 0\right\}$$
and again $(1,1)$ belongs to this set for a convenient choice of $t$. As a conclusion, the $S$-kernel contains a unique constant sequence~$b^\omega$ so the boundary sequence is $S$-automatic. 
In particular, we have shown that $\partial_{\mathbf{w}}$ is constant (for the choice of suffix $s=\varepsilon$).
\end{proof}




For the word $\mathbf{w}_1$, we can go further and prove the $S$-automaticity of its $\ell$-boundary sequence.

\begin{proposition}
Let $S = (\alpha^*\beta^*,\{\alpha,\beta\},\alpha<\beta)$ and let $\ell\ge 2$.
The $\ell$-boundary sequence $\partial_{\mathbf{w}_1,\ell}$ of the characteristic sequence $\mathbf{w}_1$ of the set of triangular numbers given in \cref{exa:astarbstar} is $S$-automatic.
\end{proposition} 
 \begin{proof}
Since $\mathbf{w}_1$ is the characteristic sequence of the triangular numbers,  its prefix of length $T_n+1$ ($n\ge 0$) ends with $1$ and contains $n+1$ occurrences of $1$; more precisely we have
 \begin{align}\label{eq: w prod consecutive 0's}
 \mathbf{w}_1 = \prod_{n\ge 0} (10^n).
\end{align}
Now consider the $\ell$-boundary sequence $\partial_{\mathbf{w}_1,\ell}$ of~$\mathbf{w}_1$.
Let $n\ge \ell$ and consider a length-$\ell$ factor $u$ of~$\mathbf{w}_1$.
Assume first that $u$ contains at most one letter $1$.
Then $u$ can take two forms.
\begin{itemize}
\item If $u=0^\ell$,  then the pairs $(u, 0^\ell)$,  $(u, 0^i 1 0^{\ell-i-1})$ for $i\in\{0,1,\ldots,\ell-1\}$ all belong to $\partial_{\mathbf{w}_1,\ell}[n]$.
\item If $u=0^i 1 0^{\ell-i-1}$ for some $i\in\{0,1,\ldots,\ell-1\}$,
then \cref{eq: w prod consecutive 0's} implies that the pairs $(u,0^\ell)$,
$(u, 0^j 1 0^{\ell-j-1})$ for $j\in\{0,1,\ldots,\ell-1\}$ all belong to
$\partial_{\mathbf{w}_1,\ell}[n]$ if $n\ge 2\ell$. Indeed, let $0\leq i,j < \ell$.
With $0^i 1 0^{\ell-i-1}\, v\, 0^j 1 0^{\ell-j-1}$ a factor of length $n + \ell$, we must have
$|v| \geq 2(i+1) - \ell - j$ and such a factor $v$ clearly
exists when this condition is satisfied. Taking $i=\ell-1$ and $j=0$ gives the maximum
length requirement $|v|\geq \ell$, so the claim follows for $n \geq 2\ell$.
\end{itemize}
Therefore, the length-$\ell$ factors containing at most one letter $1$ have the same contribution towards every boundary set in $\partial_{\mathbf{w}_1,\ell}[n]$ for $n\geq 2\ell$.
In other words,  since $\mathbf{w}_1$ has length-$\ell$ factors containing at least two letters $1$,  two boundary sets $\partial_{\mathbf{w}_1,\ell}[n]$ and $\partial_{\mathbf{w}_1,\ell}[p]$ may differ on the length-$\ell$ factors containing at least two letters $1$.

We now examine the contribution to $\partial_{\mathbf{w}_1,\ell}[n]$ of a length-$\ell$ factor $u$ containing at least two occurrences of $1$.
Due to \cref{eq: w prod consecutive 0's} again,  $u$ only appears once in $\mathbf{w}_1$,  so there is a unique factor $v$ of $\mathbf{w}_1$ such that the pair $(u,v)$ belongs to $\partial_{\mathbf{w}_1,\ell}[n]$.
From \cref{eq: w prod consecutive 0's},  one sees that long stretches of letters $0$ appear in $\mathbf{w}_1$.
Notice that $T_{\ell-2}$ is the last position (starting at $0$) of a length-$\ell$ factor containing at least two occurrences of $1$ in $\mathbf{w}_1$ (this is the factor $10^{\ell-2}1$). Then $T_{\ell-2} + \ell = T_{\ell-1}+1$.
We compute $\partial_{\mathbf{w}_1,\ell}[n]$ for $n$ ranging into two intervals: either $n$ belongs to $I_1:=[T_m +1 , T_{m+1} - T_{\ell-1}-1]$ or $I_2:=[T_{m+1} - T_{\ell-1} , T_{m+1}]$ for some $m > T_{\ell-1}$. (For $I_1$ to be non-empty,
we must have $T_{m+1} - T_{\ell-1} - 1 - T_m -1 = m -T_{\ell-1} -1 \geq 0$.)

\noindent \textbf{First interval.}
Let $u$ be a factor of $\infw{w}_1$ such that $|u|_1 \geq 2$, and let $uyv$
be the unique factor of length $n+\ell$, with $|y| = n$, starting with $u$.
We claim that $v = 0^{\ell}$.
Notice that the first letter of this particular occurrence of $v$ appears at
a position in the interval
\[
[T_m+1, T_{\ell-2} + T_{m+1}-T_{\ell-1}-1] = [T_m+1,T_{m+1}-\ell].
\]
Hence $v = 0^{\ell}$.
(See \cref{ex: triangular-2-3} for an illustration.)

\noindent \textbf{Second interval.}
Let $n=T_{m+1} - i$ and $n'=T_{m+2} - i$ for some $m$ and $i\in\{0,1,\ldots,T_{\ell-1}-1\}$.
We claim that $\partial_{\mathbf{w}_1,\ell}[n]=\partial_{\mathbf{w}_1,\ell}[n']$.
Suppose that $|u|_1\ge 2$.
Now let $y,y',v,v'\in \{0,1\}^*$ be words such that $uyv,uy'v'$ are factors of $\mathbf{w}_1$ with $|y|=n-\ell$, $|y'|=n'-\ell$ and $|v|=\ell=|v'|$.
Since $|uy|=n=T_{m+1} - i$ and $|uy'|=n'=T_{m+2} - i$,  we have $v=v'$ (note that the sequence of first difference $(T_{n+1}-T_{n})_{n\ge 0}$ goes through all positive integers). 
Therefore,  since the length of $I_2$ is constant (and equal to $T_{\ell-1}+1$),  there exists a word $w$ of length $T_{\ell-1}+1$ such that $\prod_{n \in I_2} \partial_{\mathbf{w}_1,\ell}[n] = w$.  
%
%
%

All in all,  we have shown that $\partial_{\mathbf{w}_1,\ell}$ is of the form $p \prod_{n\geq 1} c^n w$ for some word $p$,  a letter $c$,  and a word $w$ of length $T_{\ell-1}+1$.
It follows that $\partial_{\mathbf{w}_1,\ell}$ is $S$-automatic.
Indeed,  we have $\{\rep_S(T_{m}) \mid m\ge 0\}=\alpha^*$ and $\{\rep_S(T_{m+1}-i) \mid m > T_{\ell-1}\}=\alpha^i\beta^{\ell-i}\beta^*$ for $i>0$.
Since there are only finitely many values of $i$ for which the corresponding boundary sets are distinct, $\partial_{\mathbf{w}_1,\ell}$ is $S$-automatic.
 \end{proof}
 
 In the following example,  we illustrate the proof of the previous result for several values of $\ell$.

\begin{example}\label{ex: triangular-2-3}
First, consider $\ell=2$. 
Let us define the boundary sets
 \begin{align*}
 a &:=  \{(00, 00), (00,01),(00,10),(01,00),(10,00),(10,01),(10,10),(11,01) \}, \\
b &:= \{(00, 00), (00,01),(00,10),(01,00),(10,00),(01,01),(10,01),(10,10),(11,10) \},\\
c &:= \{(00, 00), (00,01),(00,10),(01,00),(10,00),(01,01),(01,10),(10,01),(10,10),(11,00) \}, \\
d &:= \{(00, 00), (00,01),(00,10),(01,00),(10,00),(01,01),(01,10),(10,01),(10,10),(11,01) \}, \\
e &:= \{(00, 00), (00,01),(00,10),(01,00),(10,00),(01,01),(01,10),(10,01),(10,10),(11,10) \}.
 \end{align*}
In this case, the last occurrence of a factor $u$ of length $2$ with $|u|_1 = 2$ appears
at position $T_{\ell-2} = 0$ since $\mathbf{w}_1=11010010001\cdots$.
The boundary set corresponding to $I_1$ is $c$ and those corresponding to $I_2$ are $d$ and $e$.
Note that these sets are all distinct, and they differ precisely on the set of the form $(11,v)$.
For $m=2 = T_{1}+1$,  the two intervals become $I_1=[4,4]$ and $I_2=[5 ,6]$,  and for $m=3$,  we obtain $I_1=[7, 8]$ and $I_2=[9,10]$.
We have $\partial_{\mathbf{w}_1,2}= a b \prod_{n\ge 1} \left( c^n d e \right)$,  for which $p=a b$,  $c=c$, and $w=d e$.

Similarly,  for $\ell=3$,  we have the last occurrence of a length-$\ell$ factor $u$ with $|u|_1 \geq 2$
appears at position $T_{\ell-2} = 1$.
For $m=T_2 + 1 = 4$,  we obtain $I_1=[11,11]$ and $I_2=[12,15]$.
Looking at the first few letters of $\mathbf{w}_1$,  one can obtain the following pairs belonging to $\partial_{\mathbf{w}_1,3}[n]$ for $n\in[11,15]$:
\[
\begin{array}{c|c}
 (u,v) & n\in[11,15] \text{ such that }(u,v)\in \partial_{\mathbf{w}_1,3}[n] \\
\hline
(110,000) & 11,12 \\
(110,001) & 13 \\
(110,010) & 14 \\
(110,100) & 15 \\
\hline
(101,000) & 11, 15 \\
(101,001) & 12 \\
(101,010) & 13 \\
(101,100) & 14
\end{array}.
\]
So,  for instance $\partial_{\mathbf{w}_1,3}[11]$ and $\partial_{\mathbf{w}_1,3}[12]$ agree on $(110,000)$ but not on $(101,000)$.
Define $c=t \cup \{(101,000),(110,000)\}$ and 
\begin{align*}
&w_1 = t \cup \{(101,001),(110,000)\}, \;\;
w_2 = t \cup \{(101,010),(110,001)\}, \\
&w_3 = t \cup  \{(101,100),(110,010)\}, \;\;
w_4 = t \cup  \{(101,000),(110,100)\},
\end{align*}
where
\begin{align*}
t 
= 
\{
&(000,000),(000,001),(000,010),(000,100),(001,000),(001,001),(001,010),(001,100),\\
&(010,000),(010,001),(010,010),(010,100), (100,000),(100,001),(100,010),(100,100).
\}
\end{align*}
From the previous table,  we have $\partial_{\mathbf{w}_1,3}[11;15]=c w_1 w_2 w_3 w_4$.

Finally,  we consider the case $\ell=5$ for which $T_{3}=6$.
For $m=T_{4}+1 = 11$,  we obtain $I_1=[67,67]$ and $I_2=[68,78]$.
In the following table are displayed the pairs $(u_j,v)$ belonging to $\partial_{\mathbf{w}_1,5}[n]$:
\[
\begin{array}{c|c|c|c|c|c}
& u_0 = 11010 & u_1 = 10100 & u_2 = 01001 & u_3 = 10010 & u_4 = 10001 \\
n \in I_1 \cup I_2 & p_0 = 0 & p_1 = 1 & p_2 = 2 & p_3 = 3 & p_4 = 6 \\
\hline
67 & 0^5 & 0^5 & 0^5 & 0^5 & 0^5\\
68 & 0^5 & 0^5 & 0^5 & 0^5 & 0^41\\
69 & 0^5 & 0^5 & 0^5 & 0^5 & 0^310\\
70 & 0^5 & 0^5 & 0^5 & 0^5 & 0^210^2\\
71 & 0^5 & 0^5 & 0^5 & 0^41 & 010^3\\
72 & 0^5 & 0^5 & 0^41 & 0^310 & 10^4\\
73 & 0^5 & 0^41 & 0^310 & 0^210^2 & 0^5\\
74 & 0^41 & 0^310 & 0^210^2 & 010^3 & 0^5\\
75 & 0^310 & 0^210^2 & 010^3 & 10^4 & 0^5\\
76 & 0^210^2 & 010^3 & 10^4 & 0^5 & 0^5\\
77 & 010^3 & 10^4 & 0^5 & 0^5 & 0^5\\
78 & 10^4 & 0^5 & 0^5 & 0^5 & 0^5\\
\end{array}.
\]
For instance,  we see that the sets $\partial_{\mathbf{w}_1,5}[n]$ for $n\in[74,78]$ differ on the pair $(u_0,v)$.
However $\partial_{\mathbf{w}_1,5}[n]$ for $n\in[67,73]$ all contain $(u_0, 0^5)$ so $u_0$ cannot tell them apart. 
To that aim,  one has to go through all columns in the previous table, therefore covering all possible values of $u_i$.
\end{example}

 \begin{example}\label{exa:astarbstar2}
   Let $S=(\alpha^*\beta^*\cup \beta^*\gamma^*,\{\alpha,\beta,\gamma\},\alpha<\beta<\gamma)$ be an abstract numeration system whose language has exactly $2n+1$ words of length~$n$. For a construction of regular languages with a specific polynomial growth, see \cite{Rigo2002}. Consider the $S$-automatic word given by the characteristic sequences of the words from the sublanguage $\alpha^*$ within $\alpha^*\beta^*\cup \beta^*\gamma^*$: $\mathbf{w}_2=1100100001000000\cdots$. This is exactly the characteristic sequence of the set of squares. This word is also obtained using the morphisms $g_2\colon 0\mapsto 01, 1\mapsto 122, 2\mapsto 2$ and $f\colon 0\mapsto \varepsilon, 1\mapsto 1, 2\mapsto 0$, $\mathbf{w}_2=f(g_2^\omega(0))$. Notice that again the ANS $S$ is non-addable, this follows from \cite[Thm.~15]{Rigo2001}.
\end{example}

\begin{proposition}\label{pro:add2}
  Let $S=(\alpha^*\beta^*\cup \beta^*\gamma^*,\{\alpha,\beta,\gamma\},\alpha<\beta<\gamma)$. 
The boundary sequence $\partial_{\mathbf{w}_2}$ of the characteristic sequence of the set of squares given in \cref{exa:astarbstar2} is $S$-automatic. In particular, it is periodic with period $babb$.
 \end{proposition}

 We provide two proofs, the first one is generic. It aims to show the finiteness of the $S$-kernel of $\partial_{\mathbf{w}_2}$ without explicitly determining the boundary sequence. The second one is less systematic but directly shows periodicity.

 \begin{proof}[Proof sketch]
   To prove that the $S$-kernel is finite, we first guess that it contains $14$ elements. We have computed prefixes of elements of the $S$-kernel with different suffixes given in \cref{tab:kpref}.
\begin{table}[h!tb]
   $$
\begin{array}{c|cccccccccccccccccccc}
\varepsilon &b & a & b & b & b & a & b & b & b & a & b & b & b & a & b & b & b & a & b & b \\
\alpha &b & b & b & b & b & b & b & b & b & b & b & b & b & b & b & b & b & b & b & b \\
\beta &a & b & a & a & b & b & b & a & b & b & a & b & b & b & a & b & a & b & b & b \\
\beta^2 &a & b & b & a & b & b & b & b & b & a & a & b & b & b & a & b & b & b & a & b \\
\beta^3 &b & b & b & b & b & a & b & b & b & a & b & b & a & b & b & b & b & b & a & b \\
\beta^4 &b & b & a & b & b & a & b & a & b & b & b & b & a & b & b & b & a & b & b & b \\
\gamma &b & b & b & b & a & b & b & a & b & b & b & b & b & a & b & b & b & b & a & b \\
\gamma^2 &b & a & b & a & b & b & b & b & a & b & b & b & a & b & b & a & b & b & b & a \\
ab &b & a & b & a & b & a & b & a & b & a & b & a & b & a & b & a & b & a & b & a \\
\alpha^2b &a & b & a & b & a & b & a & b & a & b & a & b & a & b & a & b & a & b & a & b \\
b\gamma^2 &a & a & b & b & a & a & b & b & a & a & b & b & a & a & b & b & a & a & b & b \\
\beta^2\gamma^2 &a & b & b & a & a & b & b & a & a & b & b & a & a & b & b & a & a & b & b & a \\
\beta^3\gamma^2 &b & b & a & a & b & b & a & a & b & b & a & a & b & b & a & a & b & b & a & a \\
\beta^4\gamma^2 &b & a & a & b & b & a & a & b & b & a & a & b & b & a & a & b & b & a & a & b \\
\end{array}$$
   \caption{Prefixes of elements of the form $\partial_{\mathbf{w}_2} \circ s$.}
   \label{tab:kpref}
 \end{table}   
 For the system $S$, since values are given by positions within the genealogically ordered language, we easily get
 $$\val_S(\alpha^i\beta^j)=(i+j)^2+j \quad\text{ and }\quad \val_S(\beta^j\gamma^k)=(j+k)^2+j+2k.$$
 The suffixes to consider are of the form
 $$\beta^k,\ \alpha^\ell \beta^k,\ \gamma^k,\ \beta^\ell \gamma^k \text{ with }k,\ell\ge 0.$$
 By \cref{lem:diff}, every boundary set contains at least $(0,0)$, $(0,1)$, $(1,0)$. The only question is therefore to determine whether $(1,1)$ belongs to some specific boundary set.  In view of \cref{tab:kpref}, we have to prove relations such as the following one (that we treat in details)
$$\partial_{\mathbf{w}_2}\circ \beta^{j}=\partial_{\mathbf{w}_2}\circ \beta^{j+4r},\quad j\in\{1,2,3,4\}, r>0.$$
Let $n\ge 0$. We still have \cref{eq:formk}
$$(\partial_{\mathbf{w}_2}\circ \beta^{j})[n]=\{\mathbf{w}_2[i]\mathbf{w}_2[i+\val_S(\alpha^k\beta^{\ell+j})]\mid i\ge 0\}$$
where $\alpha^k\beta^\ell$ is the $n$th word in $L(\beta^j)^{-1}$ with $L$ being the associated ANS language. Since $L(\beta^j)^{-1}=L(\beta^m)^{-1}=\alpha^*\beta^*$ for any $j,m>0$, we also have 
  $$(\partial_{\mathbf{w}_2}\circ \beta^{j+4r})[n]=\{\mathbf{w}_2[i]\mathbf{w}_2[i+\val_S(\alpha^k\beta^{\ell+j+4r})]\mid i\ge 0\}.$$
  By definition of the word $\mathbf{w}_2$ and \cref{lem:diff}, $(1,1)$ belongs to the above set if and only if $\val_S(\alpha^k\beta^{\ell+j+4r})$ is the difference of two squares. Modulo~$4$, a square is congruent to either~$0$ or,~$1$. Such a difference is not congruent to $2$ modulo $4$. It is straightforward to express any number belonging to the other three congruence classes as the difference of two squares: $4m+1=(2m+1)^2-(2m)^2$, $4m-1=(2m)^2-(2m-1)^2$ and $4m=(m+1)^2-(m-1)^2$. Observe that
  $$\val_S(\alpha^k\beta^{\ell+j})=(k+\ell+j)^2+\ell+j$$
and  $$\val_S(\alpha^k\beta^{\ell+j+4r})=(k+\ell+j+4r)^2+\ell+j+4r$$
  are congruent $\pmod{4}$. So $(1,1)$ belongs to the set $(\partial_{\mathbf{w}_2}\circ \beta^{j})[n]$ if and only if it belongs to $(\partial_{\mathbf{w}_2}\circ \beta^{j+4r})[n]$, leading to the conclusion. 

  The first few words in $L\beta^{-1}$ are $\varepsilon,\alpha,\beta$. To distinguish, as an example, the elements $\partial_{\mathbf{w}_2}\circ \beta$ and $\partial_{\mathbf{w}_2}\circ \beta^2$ of the $S$-kernel, it is enough to look at 
  $(\partial_{\mathbf{w}_2}\circ \beta)[2]=\{\mathbf{w}_2[i]\mathbf{w}_2[i+\val_S(\beta^2)]\mid i\ge 0\}$ and
 $(\partial_{\mathbf{w}_2}\circ \beta^2)[2]=\{\mathbf{w}_2[i]\mathbf{w}_2[i+\val_S(\beta^3)]\mid i\ge 0\}$ because $\val_S(\beta^2)\equiv 2\pmod{4}$ and $\val_S(\beta^3)\equiv 0\pmod{4}$. So the first one is $a$ and the second one is $b$ (as shown in \cref{tab:kpref}). 
 
 
 Proving the finiteness of the $S$-kernel amounts to prove relations such as:
 $$\partial_{\mathbf{w}_2}\circ \alpha\beta=\partial_{\mathbf{w}_2}\circ \alpha^{2m+1}\beta^{\{1,2\}+4n},\ \partial_{\mathbf{w}_2}\circ \alpha^2\beta=\partial_{\mathbf{w}_2}\circ \alpha^{2m}\beta^{\{1,2\}+4n},\ldots.$$
 \end{proof}
 Here is a shorter proof because, in our particular example, the boundary sequence is periodic.
 \begin{proof} Let us show that $\partial_{\mathbf{w}_2}=(babb)^\omega$. 
   We make use of \cref{lem:diff}: $\partial_{\mathbf{w}_2}[k]=b$ if and only if $k$ can be written as the difference of two squares $m^2-n^2$ with $m>n\ge 0$. With the same argument as in the previous proof, this holds if and only if $k$ is not congruent to $2$ modulo~$4$. Automaticity follows from the fact that any ultimately periodic set is $S$-recognizable for all ANS \cite[Thm.~4]{LecomteRigo2001}.
 \end{proof}

 \begin{remark}
   With \cref{exa:astarbstar,exa:astarbstar2}, we have exhibited sequences that are $S$-automatic for some non-addable numeration system $S$. One can naturally wonder if these sequences could also be $T$-automatic for another numeration system~$T$ being addable. Since the considered numeration systems have a polynomial growth, a Cobham-like result implies that if $\mathbf{w}_1$ (resp.,~$\mathbf{w}_2$) is $T$-automatic for some $T$, then $T$ must have a polynomial growth \cite[Cor.~27]{DurandRigo}. As a consequence of \cite[Thm.~15]{Rigo2001}, ANS with a polynomial growth are not addable. This means that \cref{exa:astarbstar,exa:astarbstar2} highlight words that are $S$-automatic only for some non-addable numeration systems $S$.

   As a side comment, if an addable numeration system is such that the graph of $n\mapsto T_n$ is also regular (i.e., the set of pairs $(\rep(n),\rep(T_n))$,  where the shortest representation is conveniently padded,  is a regular language), then the first order theory of $\langle \mathbb{N},+,x^2\rangle$ would be decidable. But this structure is equivalent to $\langle \mathbb{N},+,\cdot\rangle$ which is well known to have an undecidable theory.
 \end{remark}
 
To end up this short section, we consider a third example which is a small variation of the previous one.
\begin{example}
For $s \geq 3$, take the morphic word $\mathbf{w}_s=f(g_s^\omega(0))$ where
$g_s\colon 0\mapsto 01, 1\mapsto 12^s, 2\mapsto 2$ and
$f\colon 0\mapsto \varepsilon, 1\mapsto 1, 2\mapsto 0$. For a fixed $s$, the word $\infw{w}_s$
is the characteristic word of the set of numbers of the form $P_n:=\frac{n(s n - s +2)}{2}$.
For example, the word $\infw{w}_3$ is the characteristic sequence of the set of
\emph{pentagonal numbers} (\cite[\href{https://oeis.org/A000326}{A000326}]{OEIS}), $\infw{w}_4$
of the \emph{hexagonal numbers} (\cite[\href{https://oeis.org/A000384}{A000384}]{OEIS}), and
$\infw{w}_5$ of the \emph{heptagonal numbers} (\cite[\href{https://oeis.org/A000566}{A000566}]{OEIS}).
\end{example}

In the remainder of this part, we fix $s \geq 3$ and write $\infw{w} = \infw{w}_s$ for short.  Applying \cref{lem:diff}, the boundary
sequence is such that $\partial_{\mathbf{w}}[k]=b$ if and only if $k$ can be written as
\begin{equation}\label{eq:quad}
k=P_m-P_n=\frac{1}{2}(m-n) (s (m + n - 1) + 2)  
\end{equation}
for some integers $m > n\ge 0$. We say that an integer $k$ is \emph{representable} if there exist $m,n \in \N$ with $m>n$ such that the above equation holds.


\begin{proposition}
The boundary sequence $\partial_{\infw{w}}$ is aperiodic.
\end{proposition}
\begin{proof}
Assume first that $s$ is odd.
We make an observation about representable integers of a certain form.
\begin{claim}
Let $p$ be a prime number congruent to $1 \pmod{s}$.
For any $i,j \geq 0$, $s^i \cdot p^j$ is representable if and only if $p^j \geq s\frac{s^i-1}{2} + 1$.
\end{claim}
\begin{claimproof}
Notice that $p$ is an odd prime number.
Assume that $p^j \geq s\frac{s^i-1}{2} + 1$. Then there exists $n \geq 0$ such that $p^j = s\frac{s^i-1}{2} + 1 + n s$.
Setting $m = n + s^i$, we find
\begin{align*}
P_m - P_n 	&= \frac{1}{2}(m-n) (s (m + n - 1) + 2) \\
			&= \frac{1}{2}s^i(s(2n + s^i - 1)+2) \\
			&= s^i(s \frac{s^i -1}{2} + ns + 1) \\
			&=s^ip^j.
\end{align*}
Thus $s^i p^j$ is representable.

Assume then that $p^j < s\frac{s^i - 1}{2} + 1$, but towards a contradiction, that
$s^i p^j = P_m - P_n$ for some integers $m > n \geq 0$.
We thus have
\[
2 s^i p^j = (m-n)(s(m+n-1) + 2).
\]
Notice that $s(m+n-1)+2 \equiv 2 \pmod{s}$. Consequently, as $s \geq 3$, we must
have $s^i \mid m-n$. Furthermore, since $p \equiv 1 \pmod{s}$, we must have that
$2 \mid s(m+n-1) + 2 $ due to the same observation. Therefore, we have
$s(m+n-1) + 2 = 2p^{j_1}$ and $m = n + p^{j_2}s^i$ with $j_1 + j_2 = j$. Plugging the
latter into the former, we find
\[
2p^{j_1} = s(2n + s^ip^{j_2} - 1) + 2 = s(s^ip^{j_2}-1)+2 + s2n \geq s(s^i-1)+2 > 2p^j,
\]
where in the last inequality,  we have used the assumption.
This is a contradiction. Thus $s^i p^j$ is not representable, as claimed.
\end{claimproof}

Assume towards a contradiction that $\partial_{\infw{w}}$ is eventually periodic, i.e., 
$\partial_{\infw{w}} = uv^{\omega}$ for some finite words $u,v$. Let
$i \geq 1$ be such that $s^i \geq |u|$. Then the previous claim and \eqref{eq:quad}
imply $\partial_{\infw{w}}[s^i]=a$, and by assumption,
$\partial_{\infw{w}}[s^i + n|v|]=a$ for all $n\geq 0$.
Let however $p$ be a prime congruent
to $1 \pmod{|v| s}$ (and thus $p \equiv 1 \pmod{s}$) and
$p \geq s\frac{s^i-1}{2} + 1$. Note that there exist infinitely many primes of this form by
Dirichlet's theorem for primes in arithmetic progressions
(see, e.g., \cite[Thm.~7.9]{Apostol1976introduction}). Write $p = q\cdot |v|s + 1$.
Take $n = s^{i+1}q$; then we have
\[
s^i + n|v| = s^i + s^{i+1} q |v| = s^i(1 + q s|v|) = s^i p.
\]
This implies that $\partial_{\infw{w}}[s^i + n|v|] = b$ by the above claim together with
\eqref{eq:quad}. This contradiction shows that $\infw{w}$ is aperiodic when $s$ is odd.

Assume then that $s$ is even, say $s = 2t$ with $t \geq 2$.
Then we have that $k$ is representable if and only if
$k = P_n - P_m = (m-n)(t(m+n-1) + 1)$.
\begin{claim}
Let $p$ be a prime number congruent to $1 \pmod{s}$. Let $q = 1$ if $t$ is odd, otherwise let $q = t+1$.
Then, for all $i,j \geq 0$, we have that $t^i \cdot q \cdot p^j$ is representable if and only if
$p^j \cdot q \geq t(t^{i} -1) + 1$.
\end{claim}
\begin{claimproof}
If $p^j \cdot q \geq t(t^{i} -1) + 1$, then there exists $n \geq 0$ such that
$p^jq = t(t^i-1) + ns + 1$: indeed, if $t$ is odd, we have
$t(t^i-1) \equiv 0\pmod{s}$ and $p^jq = p^j \equiv 1 \pmod{s}$. If $t$ is even, then
$t(t^i-1) \equiv t \pmod{s}$ and we have $p^j q = p^j(t+1) \equiv t+1 \pmod{s}$. Now set $m = n + t^i$. We thus find
\[
P_m - P_n = (n-m)(t(m+n-1)+1) = t^i(t(2n+t^i - 1) + 1) = t^i(t(t^i-1) + s n + 1) = t^ip^jq,
\]
showing that $t^ip^jq$ is representable.

For the converse, assume again that $t^i p^jq = P_m - P_n$ but that $p^jq < t(t^{i} -1) + 1$. We thus have
\[
t^ip^jq = (m-n)(t(m+n-1) + 1).
\]
By inspection modulo $t$, we must have that $m-n = t^ip^{j_1}q_1$ and $t(m+n-1)+1 = p^{j_2}q_2$,
where $j_1 + j_2 = j$ and $q_1q_2 = q$.
We plug in $m = t^i p^{j_1}q_1 + n$ into the second term to obtain
\[
p^{j_2}q_2 = t(2n + t^ip^{j_1}q_1 -1) + 1 = 2n t + t (t^i p^{j_1}q_1 - 1) + 1 \geq t (t^i-1) + 1 > p^j q,
\]
where the last inequality is obtained by using the assumption.
This is a contradiction. Therefore $t^i p^j q$ is not representable, as was claimed.
\end{claimproof}

To conclude the proof of the proposition, assume again towards a contradiction that
$\partial_{\infw{w}} = uv^{\omega}$. Let $q = 1$ if $t$ is odd, and otherwise
let $q = t+1$. Let $i \geq 1$ be such
that $t(t^i-1) + 1 > q$ and $t^i \geq |u|$. Then by the above claim
$t^iq$ is not representable. In fact, by periodicity, we have that $t^iq + n|v|$ is
not representable for all $n\geq 0$. Let however $p$ be a prime with $p \equiv 1 \pmod{s|v|}$
(in particular $p \equiv 1 \pmod{s}$), and such that $pq \geq t(t^i-1) + 1$ (again Dirichlet's
theorem implies the existence of such a prime).
Write $p = r\cdot s|v| + 1$ and let $n = t^i q s r$.
We then have $t^iq + n|v| = t^iq + t^iq r s|v| = t^i q(1 +r s|v|) = t^i q p$,
which is a representable number by the above claim.
This contradiction shows that $\partial_{\infw{w}}$
is aperiodic.
\end{proof}

 \subsection{Non-addable systems: counterexamples}\label{ss:ce}

Our aim is to show that the boundary sequence of a $U$-automatic word is not always $U$-automatic. Here, we have special instances of abstract numeration systems which are, in particular, positional. So we refer to the sequence $U$ defining the system.  
We give two such examples.
The numeration system defined first is a variant of the base-$2$ system.
\begin{example}\label{exa:ce1}
  Take the numeration system $(U_n)_{n\ge 0}$ defined by $U_n=2^{n+1}-1$ for all $n\ge 0$. We have $0^*\rep_U(\mathbb{N})=(0+1)^*(\varepsilon+20^*)$. 
  Consider the characteristic word $\mathbf{u}$ of $U$, i.e., $\mathbf{u}[n]=1$ if and only if $n\in\{U_j\mid j\ge 0\}$.  The boundary sequence $\partial_\mathbf{u}$ starts with 
\[
a\, b\, a\, b\, a\, b\, a\, b\, a\, a\, a\, b\, a\, b\, a\, b\, a\, a\, a\, a\, a\, a\, a\, b\, a\, a\, a\, b\, a\, b\, a\, b\, a\, a\, a\, a\, a\, a\, a\, a\, a\, a\, a
\, a\, a\, a\, a \, b\, a\, a\, a
\cdots
\]
where $a:=\{(0,0),(0,1),(1,0)\}$ and $b:=\{0,1\} \times \{0,1\}$. \end{example}

One can show that the language $\{\rep_U(n) \colon \partial_{\infw{u}}[n]= b\}$
is not regular, hence:

\begin{proposition}
\label{prop:example_U}
  Let $U=(2^{n+1}-1)_{n\ge 0}$. The word $\mathbf{u}$ from \cref{exa:ce1} is $U$-automatic but its boundary sequence $\partial_\mathbf{u}$ is not $U$-automatic. 
\end{proposition}
\begin{proof}
The word $\mathbf{u}$ is trivially $U$-automatic.
By \cref{lem:diff},  we have $\partial_{\mathbf{u}}[k] = b$ if and only if $k$ is of the form $U_m-U_n=2^{m+1}-2^{n+1}$ for some $m> n \ge 0$.
Therefore $\partial_\mathbf{u}$ is $U$-automatic if and only if the set $X:= \{ U_{m+r}-U_m\mid m \ge 0, r>0 \}$ is $U$-recognizable (i.e., $\rep_U(X)$ is regular).
Set $R:=\rep_U(X)$ and 
$$ R_1:= 20^*, \quad 
 R_2:= \bigcup_{k\geq 1} 1^k 0^* \rep_U (k),\quad R_3:= \bigcup_{k\geq 1} 1^{U_k-1} \rep_U(U_k-1) 0^*= \bigcup_{k\geq 1} 1^{U_k-1} 20^{k-1} 0^*.$$
We have $R = R_1 \cup R_2 \cup R_3$ because of the following three observations. The $U$-representations of the elements in $X$ for $m=0$ and $r>0$ are given by the words in $R_1$ because, in that case, 
\[
\val_U(20^{r-1})=U_r-1=U_r-U_0.
\]
For $m>0$ and $r<U_m$, i.e., $|\rep_U (r)|\le m-1$, the $U$-representations of the elements in $X$ are given by the words in $R_2$ because
\[
\val_U(1^r 0^{m-|\rep_U (r)|-1} \rep_U (r))= U_{m+r} - U_m.
\]
Finally, the case $m>0$ and $r\ge U_m$ is handled by the words in $R_3$ since
\[
\val_U(1^{U_m-1} 20^{m+\ell-1})=U_{m+U_m+\ell} - U_m.
\]

An application of the pumping lemma shows that $R$ is not regular. By contradiction, if $R$ is regular, then $R\cap 1^*20^*$ is regular and accepted by a DFA with $t$ states. We conclude that there exist infinitely many integers $n_0<n_1<n_2<\cdots$ and a constant $C$ such that $1^{n_i}20^C$ belongs to $R\cap 1^*20^*$. This contradicts the form of the words in $R_2\cup R_3$. Consequently, $\partial_\mathbf{u}$ is not $U$-automatic.
\end{proof}

As a consequence of the previous proposition and \cref{the:kernel}, $U$ is non-addable.

\begin{remark}\label{rk:u-partialu-2aut}
One may notice that both $\infw{u}$ and
$\partial_\infw{u}$ are $2$-automatic: this follows by the B\"uchi--Bruy\`ere theorem \cite{BruyereHanselMichauxVillemaire}
from the set
\[
X:= \{ U_{m+r}-U_m\mid m \ge 0, r>0 \} = \{n \in \N \colon \partial_{\infw{u}}[n] = b\}
\]
being $2$-definable by the formula
\[
\varphi(n) := (\exists x) \, (\exists y)\,
  (x<y \wedge V_2(x)=x \wedge V_2(y)=y \wedge n=y-x),
  \]
  where $V_2(y)$ is the smallest power of $2$ occurring with a non-zero coefficient in the binary expansion of $y$.
\end{remark}

In view of the above remark, \cref{exa:ce1} could be considered as unsatisfactory. We now make use of a similar strategy but with a more complicated numeration system, for which we do not know any analogue of \cref{rk:u-partialu-2aut}. To this end, consider the non-addable numeration system from \cite[Ex.~3]{Frougny1997} or
\cite[Ex.~2]{MassuirPeltomakiRigo} defined by 
\begin{equation}\label{eq:Un}
V_0 = 1,\ V_1 = 4,\ V_2 = 15,\ V_3 = 54\quad \text{ and }\quad    V_n = 3V_{n-1} + 2V_{n-2} + 3V_{n-4}, \quad \forall \, n\ge 4.
  \end{equation}
  
  \begin{example}\label{exa:ce2}
  Consider the characteristic word $\mathbf{v}$ of $V$, i.e., $\mathbf{v}[n]=1$ if and only if $n\in\{V_j\mid j\ge 0\}$. This word is trivially $V$-automatic. 
The boundary sequence $\partial_\mathbf{v}$ starts with
\[
a\, a\, b\, a\, a\, a\, a\, a\, a\, a\, b\, a\, a\, b\, a\, a\, a\, a\, a\, a\, a\, a\, a\, a\, a\, a\, a\, a\, a\, a\, a\, a\, a\, a\, a\, a\, a\, a\, b\, a\, a\, a\, a\, a\, a\, a\, a\, a\, a\, b \cdots
\]
where again $a:=\{(0,0),(0,1),(1,0)\}$ and $b:=\{0,1\}\times \{0,1\}$.
\end{example}

Similar to the above, $\{\rep_V(n) \colon \partial_{\infw{v}}[n] = b\}$ is not regular, whence

\begin{proposition}
\label{prop:example_V}
  Let $V$ be the numeration system given by \eqref{eq:Un}. The word $\mathbf{v}$ from \cref{exa:ce2} is $V$-automatic but its boundary sequence $\partial_\mathbf{v}$ is not $V$-automatic. 
\end{proposition}
Before diving into the proof, we set the stage with some remarks of the numeration system given in \eqref{eq:Un}.
We assume that the reader has some knowledge about $\beta$-numeration systems, see, for instance \cite{RigoBook}. 

  The
  characteristic polynomial of~\eqref{eq:Un} has two real roots $\beta$ and $\gamma$ and two complex roots with modulus less than $1$.
  We have $\beta \simeq 3.61645$ and $\gamma \simeq -1.09685$. The number
  $\beta$ is neither a Pisot number nor a Salem number. It is however a Parry number, as it is readily checked that
  $d_\beta(1) = 3203$, where for any real number $x\in[0,1]$, we let $d_{\beta}(x)=c_0 c_1 \cdots$ denote the (greedy) $\beta$-expansion of $x$ satisfying $x=\sum_{i=0}^{\infty} c_i \beta^{-i-1}$ and $x-\sum_{i=0}^{j} c_i \beta^{-i-1}< \beta^{-j-1}$ for all $j\ge 0$. The quasi-greedy expansion $d_\beta^*(1)$ of $1$, defined as $\lim_{x\to 1^-} d_\beta(x)$, is then $(3202)^\omega$. Thus $V$ is a Parry numeration system such that $\rep_V(\mathbb{N})$ is regular.
   In our setting, every element in $\mathbb{Q}(\beta)$ is a polynomial of degree at most $3$ in $\mathbb{Q}[\beta]$.

    \begin{lemma}[{\cite[Lem. 2.2]{schmidt}}]\label{lem:schmidt}
    Let $x \in [0, 1) \cap \mathbb{Q}(\beta)$, and write
$      x = q^{-1} \sum_{i = 0}^3 p_i \beta^i$
    for integers $q$ and $p_i$. If $d_\beta(x)$ is ultimately periodic, then
    \begin{equation}\label{eq:gamma}
      q^{-1} \sum_{i = 0}^3 p_i \gamma^i = \sum_{i = 1}^\infty d_\beta(x)[i]\, \gamma^{-i}.
    \end{equation}
  \end{lemma}

\begin{proof}[Proof of \cref{prop:example_V}]
By \cref{lem:diff},  we have $\partial_{\infw{v}}[k] = b$ if and only if $k$ is of the form $V_{m+r}-V_m$ for some $m\ge 0$ and $r>0$. We discuss the value of $r$ modulo~$4$: 
\begin{align*}
\rep_V(\{V_{m+4j}-V_m\mid m\ge 0\}) &= (3202)^j0^*,\quad   j\ge 1\\
\rep_V(\{V_{m+4j+1}-V_m\mid m=0,1,2\}) &= (3202)^j(3+23+221),\quad  j\ge 0\\
\rep_V(\{V_{m+4j+1}-V_m\mid m\ge 3\}) &= (3202)^j(2203)0^*,\quad    j\ge 0\\
\rep_V(\{V_{m+4j+2}-V_m\mid m=0,1\}) &= (3202)^j(32+311),\quad  j\ge 0\\
\rep_V(\{V_{m+4j+2}-V_m\mid m\ge 2\}) &= (3202)^j(3103)0^*,\quad    j\ge 0.
\end{align*}
The first equality comes from the fact that $(3202)^j0^m$ is a greedy representation and 
\[
\val_V((3202)^j0^m)+V_m= \val_V((3202)^{j-1}(3203) 0^m)=V_{m+4j}.
\]
The reasoning is similar for the third and fifth equalities. For the third, we get
\[
\val_V((3202)^j(2203)0^m)+V_{m+3}= \val_V((3202)^{j-1}(3203) 0^m)=V_{m+4j+4},
\]
which means that $\rep_V(V_{m+4j+4}-V_{m+3})$ is $(3202)^j(2203)0^m$ because it is lexicographically less than $d_\beta^*(1)$ and thus a valid expansion.
Similarly, for the fifth, we have
\[
\val_V((3202)^j(3103)0^m)+V_{m+2}= \val_V((3202)^{j-1}(3203) 0^m)=V_{m+4j+4}.
\]

Finally, we prove that, for all $k$, there exists $M$ such that for all $m\ge M$ and $j\ge 0$, there exists a suffix $t_{m,j}\in \{0,1,2,3\}^*$ of length $m-k-2$ such that  
\begin{equation}\label{eq:forml of V-V}
\rep_V(V_{m+4j+3}-V_m)=(3202)^j\ \mathbf{d}[0]\cdots \mathbf{d}[k-1]\ t_{m,j},
\end{equation}
where $\mathbf{d}=3131202300020211200210312213101221120211\cdots$ is the $\beta$-expansion of $1-1/\beta^3=(18-7\beta-6\beta^2+2\beta^3)/9$. Roughly speaking, $\rep_V(V_{m+4j+3}-V_m)$ starts with $(3202)^j$ but then, for increasing values of $m$, the corresponding words share longer and longer prefixes of $\mathbf{d}$. See \cref{tab:beta3}.
\begin{table}[h!tb]
  $$\begin{array}{r|cccccccccccccccccccccccccccccccccccccc}
      m& \\
      \hline
 0&1&0&0&0\\
 1&3&2&0&0\\
 2&3&1&3&1&3\\
 3&3&1&3&1&2&2\\
 4&3&1&3&1&2&1&0\\
 5&3&1&3&1&2&0&3&1\\
 6&3&1&3&1&2&0&2&3&1\\
 7&3&1&3&1&2&0&2&3&0&1\\
 8&3&1&3&1&2&0&2&3&0&0&1\\
 9&3&1&3&1&2&0&2&3&0&0&1&0\\
 10&3&1&3&1&2&0&2&3&0&0&0&2&2\\
 11&3&1&3&1&2&0&2&3&0&0&0&2&1&0\\
 12&3&1&3&1&2&0&2&3&0&0&0&2&0&2&2\\
  13&3&1&3&1&2&0&2&3&0&0&0&2&0&2&1&3\\
  14&3&1&3&1&2&0&2&3&0&0&0&2&0&2&1&1&3\\
  \end{array}$$
  \caption{The $V$-representations of $V_{m+3}-V_m$ for $m=0,\ldots,14$.}
  \label{tab:beta3}
\end{table}
Let us first focus on the case $j=0$, i.e., on the $V$-representation of $V_{m+3}-V_m$. Let $k>0$. Proceed by contradiction and assume that for some $t< k$, $\mathbf{d}[0]\cdots \mathbf{d}[t]$ is not a greedy expansion, i.e., 
\[
V_{m+3}-V_m-\sum_{i=0}^t \mathbf{d}[i]\, V_{m+2-i}\ge V_{m+2-t}.
\]
Dividing both sides by $V_{m+3}$ and letting $m$ tend to infinity, we get
\[
1-1/\beta^3-\sum_{i=0}^t \mathbf{d}[i]/\beta^{i+1}\ge 1/\beta^{t+1},
\]
contradicting the fact that $\mathbf{d}$ is the $\beta$-expansion of $1-1/\beta^3$.
Now, for $j\ge 1$, write
\[
V_{m+4j+3}-V_{m}=\sum_{i=1}^j (V_{m+4i+3}-V_{m+4(i-1)+3})+V_{m+3}-V_m.
\]
By using the recurrence relation defining $V$, it is clear that
\[
\rep_V(V_{m+4i+3}-V_{m+4(i-1)+3})=3202\, 0^{m+4(i-1)+2}.
\]
Hence $\rep_V(V_{m+4j+3}-V_m)$ has the expected form~\eqref{eq:forml of V-V}.

We now show that $\mathbf{d}$ is not ultimately periodic. We apply \cref{lem:schmidt} for $x=1-1/\beta^3$. The left-hand side in \eqref{eq:gamma} is approximately $1.75$. Since $|\gamma|>1$, the right-hand side converges (absolutely) and the first few digits of its limit are $-3.57$. Hence $\mathbf{d}$ is not ultimately periodic.

To conclude the proof, we apply the pumping lemma to show that the language $R:=\rep_V(\{V_{m+r}-V_m\mid m\ge 0,r>0\})$ is not regular. Proceed by contradiction. Suppose that $R$ is accepted by a DFA with $\ell$ states. Then there exist words $u,v,w$ with $0<|v|\le \ell$ and $\mathbf{d}$ has $uv$ as prefix such that, for all $n$, $uv^nw$ belongs to $R$. This is a contradiction because $\mathbf{d}$ is not periodic.
\end{proof}

\begin{remark}
  In the above proof, it is interesting to note that the non-regularity of the language $R$ is really associated with $V_{m+r}-V_m$ for $r$ congruent to $3$ modulo~$4$. Indeed, we have used the fact that $d_\beta(1-1/\beta^3)$ is not ultimately periodic whereas
$d_\beta(1-1/\beta)=2203$, $d_\beta(1-1/\beta^2)=3103$ and $d_\beta(1-1/\beta^4)=3202$. 
\end{remark}

\begin{remark}
We do not know whether $\mathbf{v}$ and $\partial_{\mathbf{v}}$ are both $V'$-automatic for some numeration system $V'$.
\end{remark}

\section{The extended boundary sequences of Sturmian words}
\label{sec:Sturmian}
We give two descriptions of the $\ell$-boundary sequences of
Sturmian words (\cref{the:sturmian,prop:another-charact}) and discuss some of their word combinatorial properties.
We first recap minimal background on Sturmian words seen as codings of rotations.
For a general reference, see \cite[\S 2]{Lothaire}.
Let $\alpha$, $\rho \in \mathbb{T}:= [0,1)$ with $\alpha$ irrational. Define the \emph{rotation} of the $1$-dimensional torus $R_{\alpha}\colon \mathbb{T} \to \mathbb{T}$ by $R_{\alpha}(x) = \{x + \alpha\}$,
where $\{\,\cdot\, \}$ denotes the fractional part. Let 
$I_0 = [0,1-\alpha)$ (or $I_0 = (0,1-\alpha]$) and
$I_1 = \mathbb{T}\setminus I_0$. 
(The endpoints of $I_0$ will not matter in the forthcoming arguments.) 
Define the coding $\nu \colon \mathbb{T} \to \{0,1\}$ by $\nu(x) = 0$ if $x \in I_0$, otherwise $\nu(x) = 1$. We define the word
$\mathbf{s}_{\alpha,\rho}$ by
$\mathbf{s}_{\alpha,\rho}[n] = \nu(R_{\alpha}^n(\rho))$, for all $n\ge 0$.
We call $\alpha$ the {\em slope} and $\rho$ the {\em intercept} of $\mathbf{s}_{\alpha,\rho}$.
The \emph{characteristic Sturmian word of slope $\alpha$} is $\mathbf{s}_{\alpha,\alpha}$.

\subsection{A description of the extended boundary sequence}\label{subsec:first}

In the following, a \emph{sliding block code of length $r$} is a mapping
$\mathfrak{B}\colon A^{\N} \to B^{\N}$ defined by
$\mathfrak{B}(\infw{x})[n] = \mathcal{B}(\infw{x}[n]\cdots \infw{x}[n + r-1])$
for all $n\ge 0$ and some $\mathcal{B} \colon A^{r} \to B$. Let $T \colon A^{\N} \to A^{\N}$
denote the shift map $Tx_0x_1x_2\cdots = x_1x_2\cdots$.

\begin{theorem}
\label{the:sturmian}
For a Sturmian word $\infw{s}$ of slope $\alpha$ (and intercept $\rho$) and $\ell \geq 1$, the (shifted) $\ell$-boundary sequence $T\partial_{\infw{s},\ell}$
is obtained by a sliding block code of length $2\ell$ applied to
the characteristic Sturmian word of slope $\alpha$.
\end{theorem}

To prove the theorem we develop the required machinery. For a word~$u=u_0\cdots u_{\ell-1}$, we let $I_u= \bigcap_{i=0}^{\ell-1}R_\alpha^{-i} (I_{u_i})$.
It is well known that $u$ occurs at position $i$ in $\mathbf{s}_{\alpha,\rho}$ if and only if $R_{\alpha}^{i}(\rho) \in I_u$. These intervals of factors
of length $\ell$ can also be described as follows: order the set
$\bigr\{ \{-j\alpha\} \bigr\}_{j = 0}^{\ell}$
as $0 = i_0 <i_1<i_2<\cdots < i_\ell$.
For convenience, we set $i_{\ell+1}=1$.
If the $\ell+1$ factors of length~$\ell$ of the Sturmian word
$\mathbf{s}_{\alpha,\rho}$ are lexicographically ordered as $w_0<w_1<\cdots <w_\ell$, then
$I_{w_j}=[i_j,i_{j+1})$ for each $j\in\{0,\ldots,\ell\}$. From the
following claim it is evident that the intercept $\rho$ plays no further role in our considerations. (This also follows from the fact that two Sturmian words have the same set of factors if and only if they have the same slope.)

\begin{claim}
Let $n\ge\ell$ and $u$, $v$ be length-$\ell$ factors of $\mathbf{s}_{\alpha,\rho}$. Then $ (u,v)\in \partial_{\mathbf{x},\ell}[n]$ if and only if
$R_\alpha^n(I_u) \cap I_v\neq\emptyset$.
\end{claim}
\begin{claimproof}
We have $(u,v)\in \partial_{\mathbf{x},\ell}[n]$
if and only if there exists $i$ such that $R_{\alpha}^{i}(\rho) \in I_u$ and
$R_{\alpha}^{i+n}(\rho) \in I_v$, or equivalently,
$R_{\alpha}^{i}(\rho) \in I_u \cap R_{\alpha}^{-n}(I_v)$. Notice that the intersection is a finite union of (possibly empty) intervals. Since the set $(R_{\alpha}^{i}(\rho))_{i\in \N}$ is dense in $\mathbb{T}$, it follows that
there exists $i$ such that $R_{\alpha}^{i}(\rho) \in I_u \cap R_{\alpha}^{-n}(I_v)$ if and only if
$I_u \cap R_{\alpha}^{-n}(I_v) \neq \emptyset$. The claim
follows by applying the isomorphism $R_{\alpha}^n$ to the intersection.
\end{claimproof}

The endpoints of $I_u$ are of the form $i_j$ and $i_{j+1}$ for some $j\in\{0,\ldots,\ell\}$. Hence, for $n \geq \ell$, 
the set of pairs belonging to $\partial_{\mathbf{x},\ell}[n]$ is determined by the positions of the rotated endpoints $R_\alpha^n(i_j)$ within the intervals $I_{w_k}$.
Notice that each rotated endpoint $R_{\alpha}^n(i_j)$ always
lies in the interior of some $I_{w_k}$ whenever $n > \ell$. When $n = \ell$,
we have $R_{\alpha}^n(\{-\ell\alpha\}) = 0$, which is an endpoint of
one of the intervals $I_{w_k}$.
For the time being we assume $n > \ell$, and return to the
case $n=\ell$ in \cref{prop:first-letter}.
Now, for example, if $R_\alpha^{n}(i_j) \in I_{w_k}$ then we have
$(w_j,w_k)$, $(w_{j-1},w_k) \in \partial_{\mathbf{x},\ell}[n]$
(if $j=0$, $w_{j-1}$ is replaced with $w_\ell$). Determining the boundary sets can be quite an intricate exercise; see \cref{exa:ffib}.

An alternative to considering the positions of the points $R_\alpha^{n}(i_j)$
within the intervals $I_{w_k}$ is to consider the positions
of the points $R_\alpha^{n}(\{-j \alpha\})$ within the intervals
$I_{w_k}$---the only difference is the order of enumeration.
For each $n > \ell$, there is a map $\sigma = \sigma_n \in T_{\ell}$,
where $T_{\ell}$ is the set of mappings from $\{0,\ldots,\ell\}$
to itself, such that
\begin{equation}\label{eq:realizable-constellations}
R_{\alpha}^n(\{-j\alpha\}) \in I_{w_{\sigma(j)}} \quad \forall\, j\in \{0,\ldots,\ell\}.
\end{equation}
The realizable such configurations in \eqref{eq:realizable-constellations} are called \emph{constellations}.
These points, when ordered according to the
$i_j$'s, determine the boundary set
$\partial_{\infw{s},\ell}[n]$ as described above.
    See \cref{exa:ffib} (and \cref{exa:intricated}) for an illustration of the construction.

\begin{definition}\label{def:r}
Let $\sigma \in T_{\ell}$ be such that
\eqref{eq:realizable-constellations} holds for some $n\in \N$. We
define $\partial_{\sigma}\in 2^{A^\ell\times A^\ell}$ as the
boundary set corresponding to any constellation inducing
$\sigma$.
\end{definition}
It is now evident that if $\sigma_n = \sigma_m =: \sigma$, then $\partial_{\infw{s},\ell}[n] = \partial_{\sigma} = \partial_{\infw{s},\ell}[m]$.

\begin{example}\label{exa:ffib}
The Fibonacci word $\mathbf{f}$ is $\mathbf{s}_{\alpha,\alpha}$ for $\alpha=(3-\sqrt{5})/2\simeq 0.382$. In \cref{fig:rotations}, the outer circle shows the partition with the interval $I_{w_0}$, \ldots, $I_{w_\ell}$ and the inner circle shows the positions of the points $R_\alpha^n(\{-j\alpha\})$ for $\ell=4$ and $n=17$. 
The corresponding words $w_0,\ldots,w_\ell$ are written next to their interval.
Here $\sigma_n$ is defined by
$(0,1,2,3,4) \mapsto (2,0,3,1,4)$. For any constellation inducing $\sigma_n$, we see the pairs belonging to $\partial_{\sigma_n} = \partial_{\mathbf{f},4}[17]$ from \cref{fig:rotations}: the inner intervals (obtained from the outer intervals by applying $R_\alpha^{17}$) give the prefix matching the suffix of the overlapping outer intervals, in clockwise order:
\[
   \smatrix{0010 \\0101}, \smatrix{0010 \\ 1001}, \smatrix{0100 \\1001}, \smatrix{0100\\1010}, \smatrix{0101 \\ 1010}, \smatrix{0101\\ 0010}, \smatrix{1001 \\ 0010}, \smatrix{1001\\0100}, \smatrix{1010 \\ 0100}, \smatrix{1010 \\ 0101}.
\]
Coming back to the introductory \cref{exa:intro}, the five sets $a_1$, \ldots, $a_5$ correspond to the situations depicted from left to right in \cref{fig:exa2fib}. For instance, in the fourth picture, we understand why $10$ is a prefix belonging to three pairs in $a_4$: the red inner interval intersects the three outer intervals of the partition. The situation is similar in the fifth picture where $01$ is the prefix of three pairs in $a_5$. It is however not the case with the first three sets/pictures.
\begin{figure}
\centering
\begin{minipage}{.39\textwidth}
  \centering
  \includegraphics[width=5.4cm]{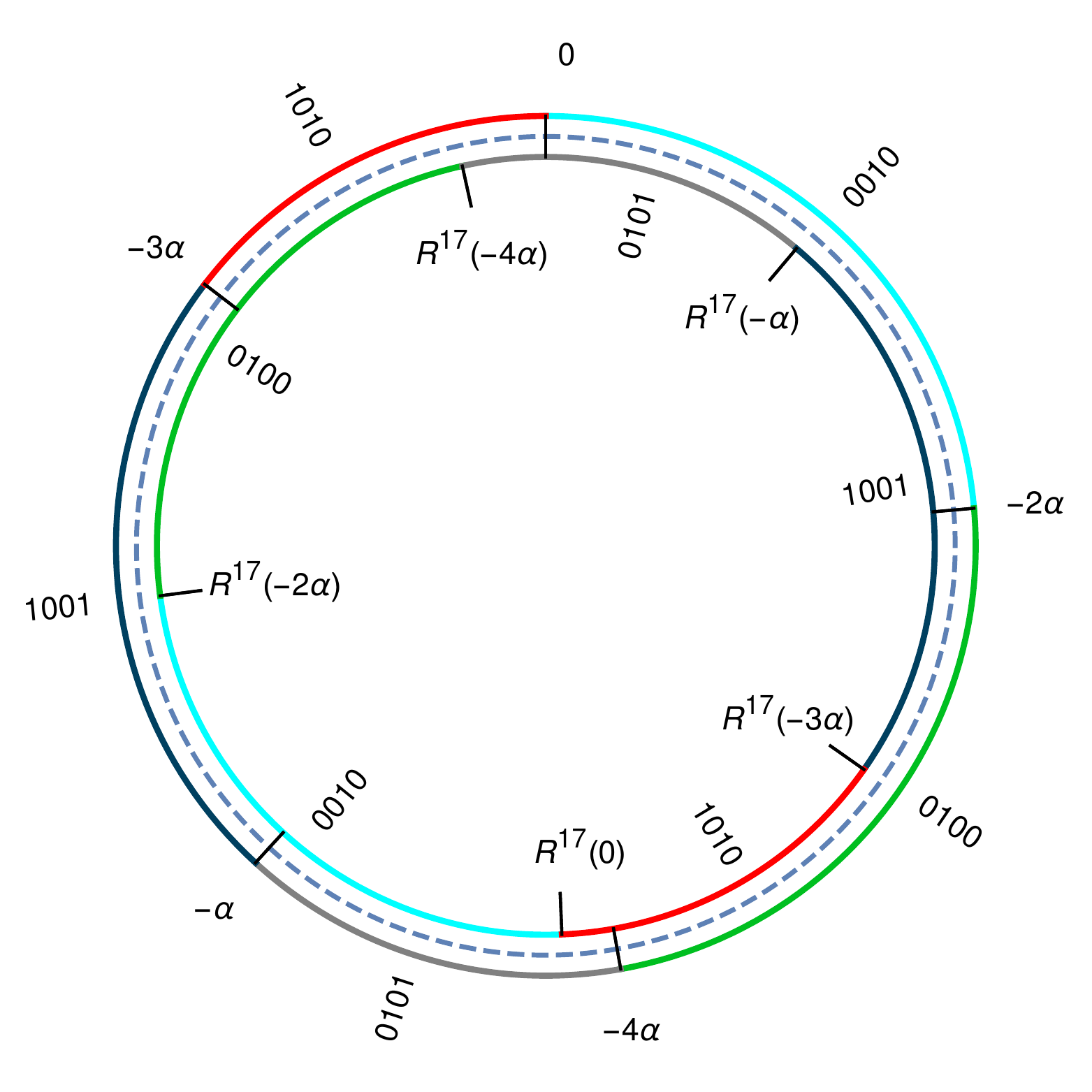}
  \caption{A constellation for\\ $\alpha=(3-\sqrt{5})/2$, $\ell=4$ and $n=17$.}
  \label{fig:rotations}
\end{minipage}%
\begin{minipage}{.61\textwidth}
  \centering
  \includegraphics[height=5.35cm]{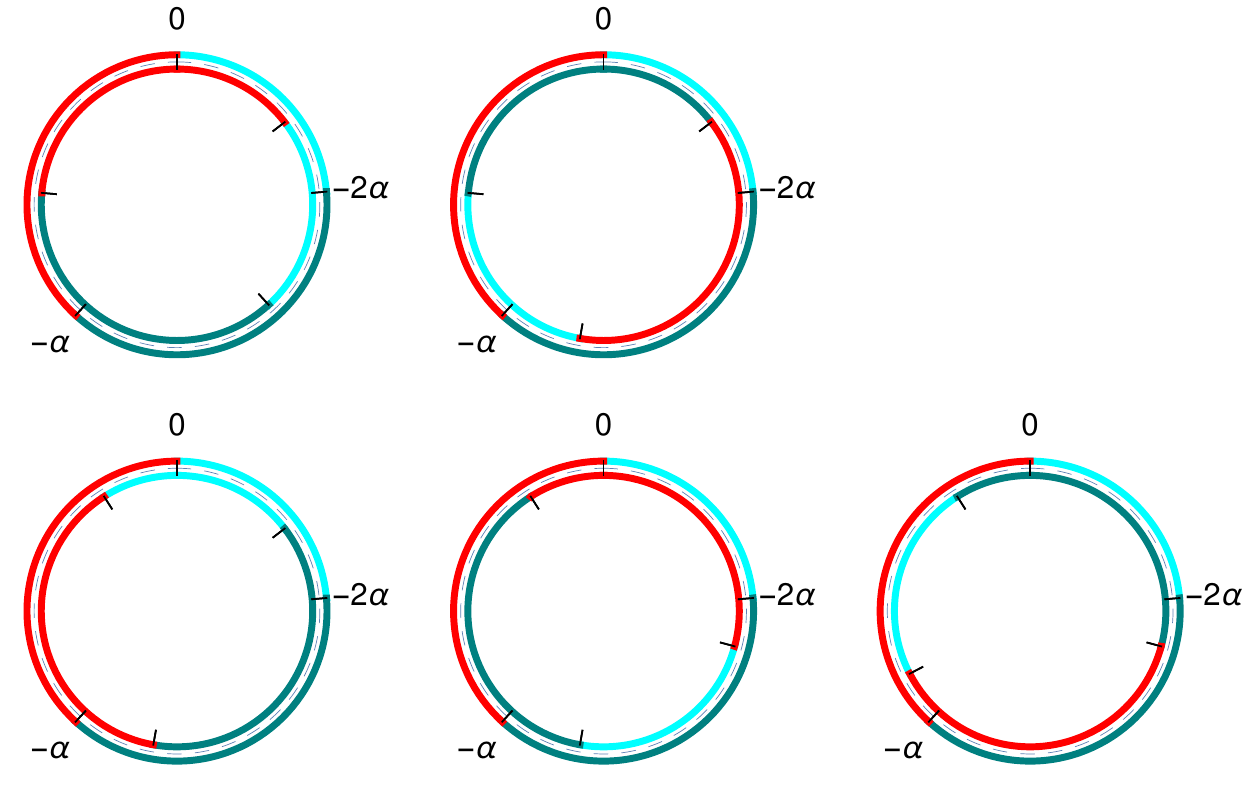}
  \caption{Some constellations for $\alpha=(3-\sqrt{5})/2$ and $\ell=2$ inducing the five maps $\sigma_n$ sending $(0,1,2)$, resp., to\\
    $(0,2,1)$, $(1,0,2)$, $(2,1,0)$, $(1,2,1)$, $(2,1,2)$.}
  \label{fig:exa2fib}
\end{minipage}
\end{figure}
\end{example}

We give an accompanying example to \cref{exa:ffib} for the reader to clarify the notion on constellations.

\begin{example}\label{exa:intricated}
What matters to determine the pairs belonging to the $\ell$-boundary
sequence are the non-empty intersections of the form
$R_\alpha^n(I_u) \cap I_v$. There are situations where
$R_\alpha^n(I_u) \subset I_v$ or $R_\alpha^n(I_u) \supset I_v$, whence
$\sigma_n$ is neither injective nor surjective.
   For instance, this is the case for the last two constellations in \cref{fig:exa2fib} (we have $\sigma_n$ equals $(0,1,2) \mapsto (1,2,1)$,
  and $(0,1,2) \mapsto (2,1,2)$, respectively). With $\alpha=(\pi-3)/2\simeq 0.0708$ and $\ell=5$, the partition of $\mathbb{T}$ is made of $5$ short intervals of length~$\alpha$ and one large interval of length $1-5\alpha>0.5$. 
 In \cref{fig:pi}, we see that five or four ``short'' rotated intervals are included in the same large interval (for $n$ equal to $21$ and $10$ respectively). In particular, counting the number of matching pairs of colors around the circle, we see that
  $\partial_{\mathbf{x},5}[5]= \partial_{\mathbf{x},5}[21]$ with cardinality $11$ and $|\partial_{\mathbf{x},5}[10]|=12$. Contrarily to \cref{exa:ffib} and \cref{fig:rotations} where each prefix and suffix belong to two pairs, here one prefix (corresponding to the large interval) belongs to six pairs of the boundary and the other prefixes belong to one pair (or two for one short interval in the constellation on the right of \cref{fig:pi}).
  \begin{figure}[h!t]
    \centering
    \includegraphics[width=.32\linewidth]{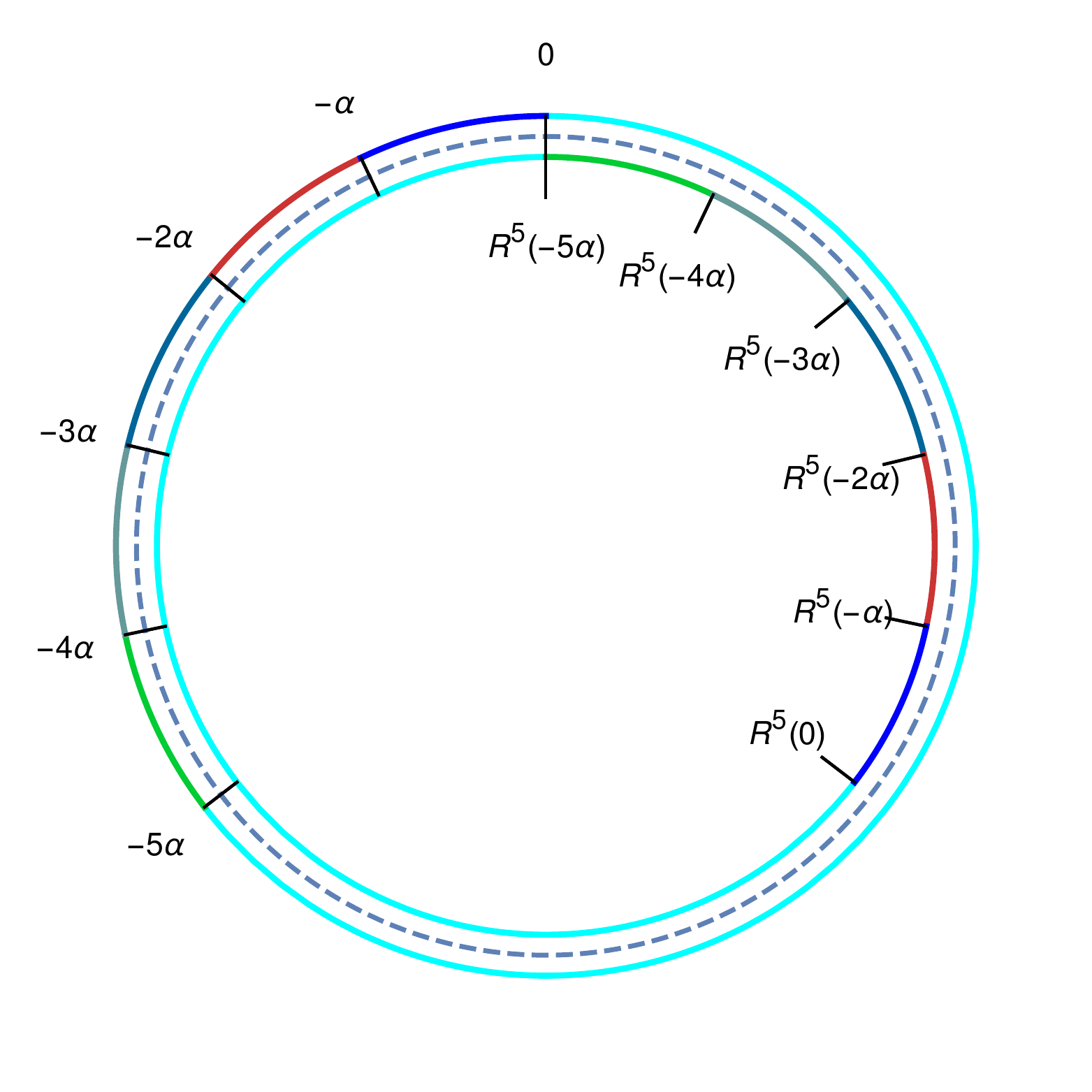}
        \includegraphics[width=.32\linewidth]{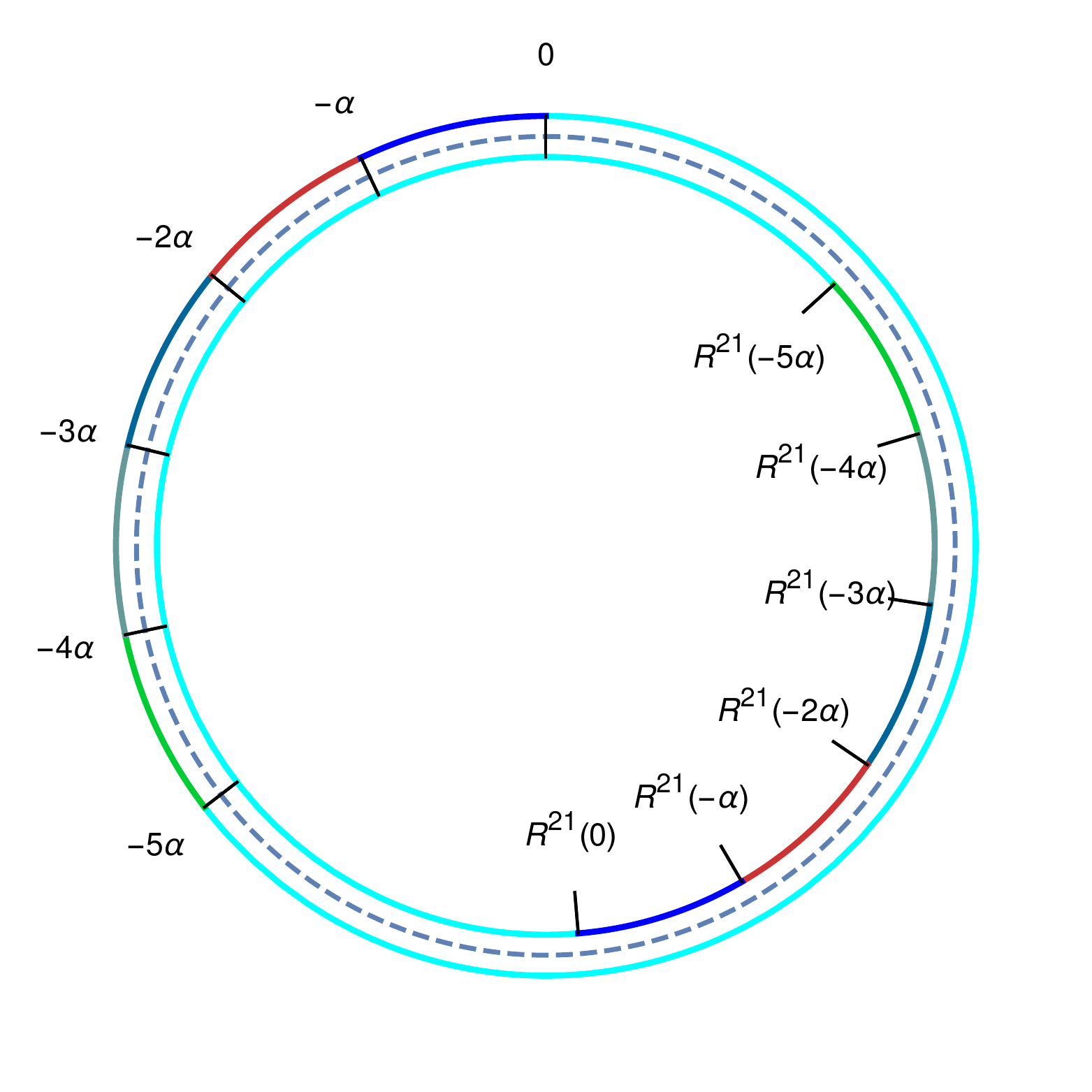}
        \includegraphics[width=.32\linewidth]{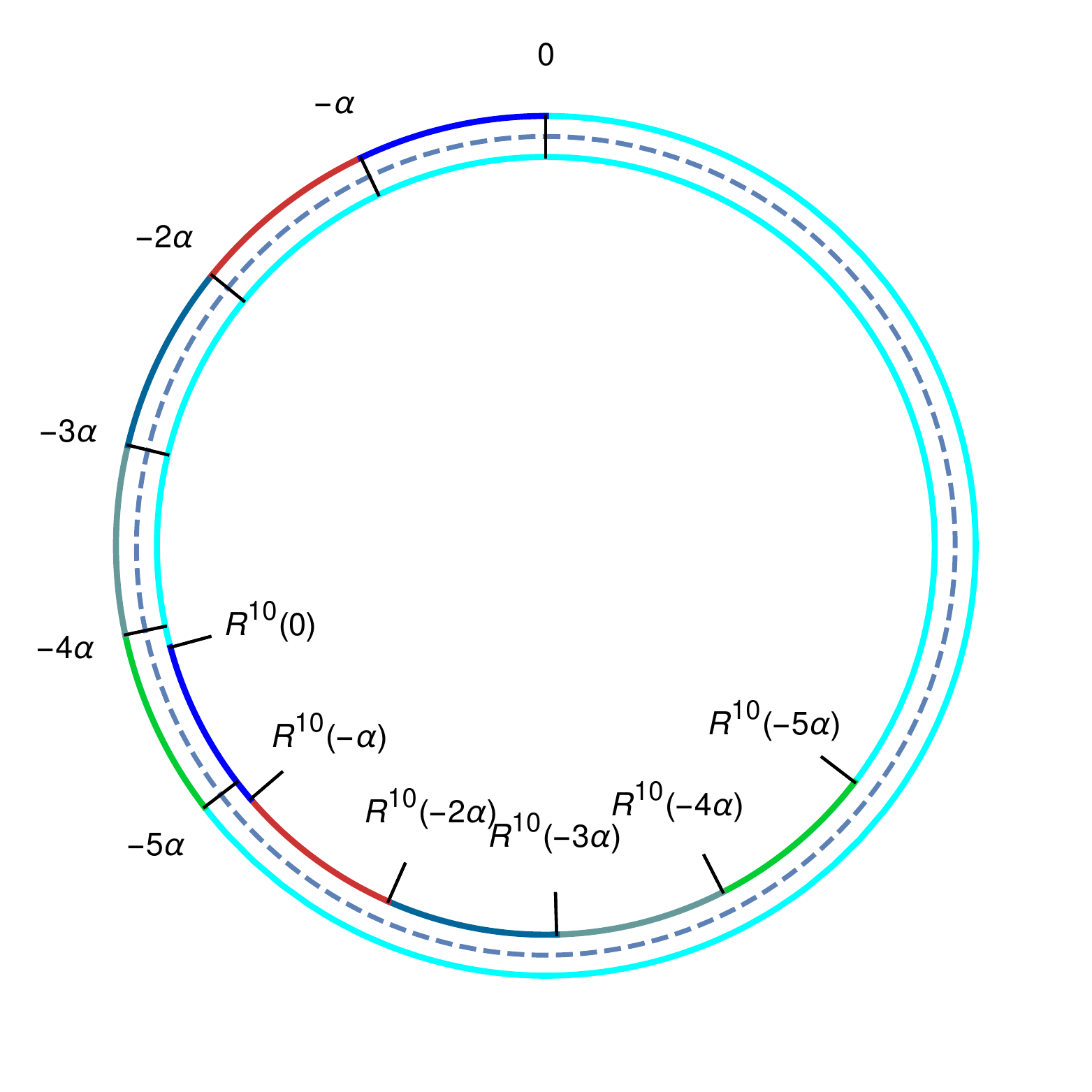}
    \caption{Constellations for $\alpha=(\pi-3)/2$, $\ell=5$, and $n = 5$, $21$, $10$.}
    \label{fig:pi}
  \end{figure}
\end{example}

\begin{remark}\label{rem:necessary}
It is possible that $\partial_{\sigma} = \partial_{\sigma'}$ for distinct maps $\sigma$, $\sigma' \in T_{\ell}$. Indeed,
for the Fibonacci word and $\ell=1$, we have equality for the identity mapping $\text{id}$ and $\sigma \colon (0,1) \mapsto (1,0)$;
in this case $\partial_{\text{id}} = \partial_{\sigma} = 
\{0,1\}\times \{0,1\}$.
So two constellations inducing different maps in $T_{\ell}$ lead to
the same set of boundary pairs. (See however \cref{lem:distinctSets}.)
\end{remark}

\begin{definition}
Let $\infw{r}$ be the rotation word defined by $\infw{r}[n] = \eta(R_{\alpha}^n(\alpha))$ for all $n\ge 0$, where $\eta \colon \mathbb{T} \to \{0,\ldots,\ell\}$ is defined by $\eta(x) = j$ when $x \in I_{w_j}$ (recall $I_{w_i}$ corresponds to the $i$th factor of length $\ell$).
\end{definition}
We have that $\infw{r}[n] = j$ if and only if the characteristic Sturmian word $\infw{s}_{\alpha,\alpha}$ has the length-$\ell$ factor $w_j$
occurring at position $n$.


\begin{proof}[Proof of \cref{the:sturmian}]
Notice that by definition, the word $\infw{r}$ defined in \cref{def:r} is obtained by a sliding block code of length $\ell$ of the characteristic Sturmian word $\infw{s}_{\alpha,\alpha}$. We show that $T\partial_{\infw{s},\ell}$ is obtained from
$\infw{r}$ by a sliding block code of length $\ell+1$. The claim then
follows since the composition of sliding block codes of length
$r$ and $r'$, respectively, is a sliding block code of length $r + r' - 1$.

Let $n > \ell$. Consider the factor of length $\ell+1$ of $\mathbf{r}$ occurring at position $m = n - \ell - 1 \geq 0$:
by definition we have $\mathbf{r}[m] \mathbf{r}[m+1] \cdots \mathbf{r}[m+\ell] = u_0u_1\cdots u_{\ell}$ if and only if
$u_{\ell - j} = \eta(R_{\alpha}^{m+ \ell - j}(\alpha))
	= \eta(R_{\alpha}^{m+\ell + 1}(\{-j\alpha\}))$,
	for each $j\in\{0,\ldots,\ell\}$.

This is equivalent to $R^{m+\ell + 1}(\{-j\alpha\}) \in I_{u_{\ell - j}}$ for each $j\in\{0,\ldots,\ell\}$. There thus exists a mapping
$\sigma \in T_{\ell}$ such that
$R_{\alpha}^{m+\ell + 1}(\{-j\alpha\}) \in I_{w_{\sigma(j)}}$, whence
$\partial_{\infw{s},\ell}[n]=\partial_{\infw{s},\ell}[m+\ell + 1] = \partial_{\sigma}$.
We conclude that the factor of length $\ell + 1$ appearing at position
$m$ in $\infw{r}$ determines the boundary set
$\partial_{\infw{s},\ell}[n]$. Letting the mapping
$\mathcal{B} \colon \{0,\ldots,\ell\}^{\ell + 1} \to T_{\ell}$ capture this relation, we may define an associated sliding block code $\mathfrak{B}$ of length $\ell + 1$ such that
$\mathfrak{B}(\infw{r}) = T\partial_{\infw{s},\ell}$.
\end{proof}

\begin{example}\label{exa:fib}
  We apply \cref{the:sturmian} to the Fibonacci word $\infw{f}$. Take $\alpha=(3-\sqrt{5})/2$, $\ell=1$, $I_0=[0,1-\alpha)$ and $I_1=[1-\alpha,1)$. Then the rotation word $\infw{r}$ associated with the partition
  $\{I_0,I_1\}$, slope $\alpha$, and intercept $\alpha$ is $\infw{s}_{\alpha,\alpha}$ by definition, which happens to be the Fibonacci word $\infw{f}$.
  We have
  $
  \mathbf{f}= 0\ 1\ 0\ 0\ 1\ 0\ 1\ 0\ 0\ 1\ 0\ 0\ 1\ 0\ 1\ 0\ 0\ 1\ 0\ 1\ 0\ 0\ 1\ 0\ 0\ 1\ 0\ 1\ 0\ 0\ 1 \cdots.
  $
Recall from the construction that the length-$2$ factors of the rotation
word determine the boundary sets. The three length-$2$ factors of $\infw{f}$ are $01$, $10$, and $00$ occurring at positions $m = 0$, $1$, and $2$,
respectively. We get the three maps $\sigma_{m+2}\in T_1$ defined by $(0,1) \mapsto (1,0)$,
$(0,1) \mapsto (0,1)$, and $(0,1) \mapsto (0,0)$, respectively. We deduce that an occurrence of $01$ or $10$ corresponds to the boundary set
$b := \{0,1\} \times \{0,1\}$, and $00$ to
$a := \{(0,0),(0,1),(1,0)\}$.
We may therefore define $\mathcal{B} \colon 01, 10 \mapsto b$, $00 \mapsto a$ and the associated sliding block code $\mathfrak{B}$ of length~$2$;
applying $\mathfrak{B}$ to $\infw{f}$, we get
\begin{align*}
& \mathfrak{B} ( (01)(10)(00)(01)(10)(01)(10)(00)(01)(10)(00)(01)(10)(01)(10)(00)(01)(10)(01) \cdots )\\ 
&= \hskip6pt  b \hskip14pt b \hskip13pt a \hskip13pt b \hskip13pt
b \hskip14pt b \hskip13pt b \hskip13pt a \hskip13pt b \hskip14pt b \hskip13pt a \hskip13pt b \hskip13pt b \hskip14pt b \hskip13pt b \hskip13pt 
  a \hskip13pt b \hskip14pt b \hskip14pt b \cdots, 
\end{align*}
  which indeed gives back \cref{exa:intro} after prepending the letter $a$.
\end{example}

We next discuss the first element
$\partial_{\mathbf{s},\ell}[\ell]$
of the (extended) boundary sequence.
Notice that the set is in one-to-one correspondence with the factors
of length $2\ell$, and thus has cardinality $2 \ell  + 1$.
The points $\{-j\alpha\}$ and $R_{\alpha}^{\ell}(\{-j\alpha\})$, $j \in\{0,\ldots,\ell\}$, on the torus
still determine the boundary set, but notice that there are only $2\ell + 1$
distinct pairs. The following proposition describes rather precisely
how the first element appears in the boundary sequence. 

\begin{proposition}
\label{prop:first-letter}
For a Sturmian word $\infw{s}$, the boundary set
$\partial_{\infw{s},\ell}[\ell]$ appears infinitely often in
$\partial_{\infw{s},\ell}$ if and only if
$0^{2\ell}$ or $1^{2\ell}$ appears in $\infw{s}$. Otherwise it appears exactly once. 
\end{proposition}

In what follows, for $x \in \mathbb{T}$, we define $\| x \| = \min\{x,1-x\}$, whence $\|x\| < 1/2$ for irrational~$x$. It is not hard to show that $0^k$ or $1^k$ appears in a Sturmian word of slope $\alpha$ if and only if $k\|\alpha\| < 1$.

\begin{proof}[Proof of \cref{prop:first-letter}]
Assume first that $\ell \|\alpha\| > 1/2$. Consider the set $R_{\alpha}^n(I_u) \cap I_v$ for some length-$\ell$ factors $u$, $v$, and $n\geq \ell$.
We claim that it is an interval whenever it is non-empty. If it is not, then the intersection is a union of two intervals: without loss of generality $|I_u| > 1-|I_v|$, and $R_{\alpha}^n(I_u)$ intersects $I_v$ from both ends, but does not contain $I_v$ entirely. Notice that the intervals corresponding to length-$\ell$ factors have length at most $\|\alpha\|$ whenever $\ell \|\alpha\| > 1$. Since in that case we get the contradiction
$|I_u| > 1 - |I_v| \geq 1 - \|\alpha\| > \|\alpha\|$,
we must have $\ell \|\alpha \| < 1$. But now we know that the intervals have two admissible lengths,
namely $\| \alpha \|$ and $1 - \ell\|\alpha\|$ (compare to the non-rotated points in \cref{fig:pi} for an illustration). Now if $1-\ell|\alpha\|$
is the largest of the two, we have a contradiction
$|I_u| > 1-|I_v| = 1-\|\alpha\| \geq 1 - \ell\|\alpha\| = |I_u|$. Conversely, we get the contradiction
$|I_u| > 1-|I_v| = 1-(1-\ell \|\alpha \|) = \ell \|\alpha\| > 1/2 > \|\alpha\| = |I_u|$. We conclude that for any length-$\ell$ factors $u$, $v$ of $\infw{s}$, the set $R_{\alpha}^n(I_u) \cap I_v$ is an interval or is empty. This implies that the boundary set $\partial_{\infw{x},\ell}[n]$ contains $2\ell + 2$ elements whenever $n > \ell$. Thus $\partial_{\infw{x},\ell}[\ell]$ occurs only once in the $\ell$-boundary sequence due to a cardinality argument.

Assume then that $\ell \|\alpha \| < 1/2$. Without loss of generality
we can assume $\alpha < 1/2$. It is straightforward to verify that
$\partial_{\infw{x},\ell}[\ell] = \{(0^i10^{\ell - i - 1},0^{\ell})\}_{i = 0}^{\ell - 1} \cup \{(0^{\ell},0^i 1 0^{\ell - i - 1})\}_{i = 0}^{\ell - 1} \cup \{ (0^\ell,0^\ell) \}$.
See, for instance, the first picture in \cref{fig:pi}.
This same set is obtained for those $n$ for which $0 < R_{\alpha}^n(\{-\ell\alpha\}) <  R_{\alpha}^n(0) < 1- \ell \|\alpha\|$:
two of the  $2\ell + 2$ intervals correspond to the boundary pair
$(0^\ell,0^\ell)$, namely the intervals $[0,R_{\alpha}^n(\{-\ell\alpha\})$ and $[R_{\alpha}^n(0), 1- \ell \|\alpha\|)$.
Again see \cref{fig:pi} for an illustration: $n=21$ in the second picture satisfies the previous condition while $n=10$ in the third picture does not.
\end{proof}

Notice that either $00$ or $11$ appears in a Sturmian word $\infw{s}$, so the above implies that the
first letter of the ($1$-)boundary sequence~$\partial_{\infw{s}}$
always appears infinitely often in the sequence.
Returning to \cref{exa:intro}, since $0^4$ does not appear in the Fibonacci word, the letter $a_0$ appears only once in $\partial_{\infw{f},2}$. 

We conclude with the immediate corollary of \cref{the:sturmian,prop:first-letter}; here we say that a word
$\infw{w}$ is \emph{uniformly recurrent} if each of its factors occurs infinitely often within bounded gaps (the distance between two consecutive occurrences depends on the factor). It is known that, e.g., Sturmian words are uniformly recurrent.

\begin{corollary}
For any Sturmian word $\infw{s}$, the shifted sequence $T\partial_{\infw{s},\ell}[n]$ is uniformly recurrent. The sequence $\partial_{\infw{s},\ell}$ is uniformly recurrent if and only if $0^{2\ell}$ or $1^{2\ell}$ appears in $\infw{s}$.
\end{corollary}

\subsection{Another description of the extended boundary sequence}\label{subsec:second}

We give another description of the $\ell$-boundary sequences of Sturmian words when $\ell \geq 2$.
For any irrational number $\alpha \in (0,1)$ there is a unique infinite continued fraction expansion
\[
 \alpha = [0; a_1 , a_2, a_3, \ldots] := \dfrac{1}{a_1 + \dfrac{1}{a_2 + \dfrac{1}{a_3 + \ldots}}},
\]
where $a_n \geq 1$ are integers for all $n\geq 1$. Then the
characteristic Sturmian word $\infw{s}_{\alpha,\alpha}$ of slope
$\alpha$ equals $\lim_{k\to\infty} S_k$,
where $S_{-1} = 1$, $S_0 = 0$, $S_1 = S_0^{a_1-1}S_{-1}$, and $S_{k+1} = S_{k}^{a_{k+1}}S_{k-1}$ for all $k\geq 1$ \cite[Chap.~9]{AS}.
The main result of this part is the following.

\begin{proposition}
\label{prop:another-charact}
Let $\infw{s}$ be a Sturmian word of slope $\alpha = [0;a_1+1,a_2,\ldots]$.
For each $\ell \geq 2$, there exists $k_{\ell} \in \N$ such that for any $k \geq k_{\ell}$ there is a morphism $h_{k,\ell}$ such
that $T\partial_{\mathbf{s},\ell} = h_{k,\ell}(\infw{s}_{\beta_k,\beta_k})$, where
$\beta_k = [0;a_{k+1}+1,a_{k+2},\ldots]$.
\end{proposition}

\begin{proof}
Let $(S_j)_{j\geq -1}$ be the sequence associated to
the slope $\alpha$. Let then $k$ be an integer such that
$|S_{k}S_{k-1}| \geq 2\ell+1$.
Hence $\infw{s}_{\alpha,\alpha}$ is a product
of $S_k$ and $S_{k-1}$. It is now evident that with
$\beta = [0;a_{k+1}+1,a_{k+2},\ldots]$, we have that
$g(\infw{s}_{\beta,\beta}) = \infw{s}_{\alpha,\alpha}$,
where $g$ is defined by
$g\colon 0\mapsto S_k$, $1\mapsto S_{k-1}$.
 Let
$X = \pref_{2\ell-1}(S_{k}S_{k-1})$. We also have that
$X = \pref_{2\ell-1}(S_{k-1}S_{k})$, as $S_kS_{k-1}$ and
$S_{k-1}S_k$ are known to differ in only the last two
letters~\cite[Thm.~9.1.11]{AS}. We have that
$X$ is a prefix of both $S_kX$ and $S_{k-1}X$:
\[
\pref_{2\ell -1}(S_kX) = \pref_{2\ell - 1}(S_k\pref_{2\ell-1}(S_{k-1}S_{k})) = \pref_{2\ell-1}(S_kS_{k-1}) = X
\]
and
\[
\pref_{2\ell -1}(S_{k-1}X)
	= \pref_{2\ell - 1}(S_{k-1}\pref_{2\ell-1}(S_{k}S_{k-1}))
	= \pref_{2\ell-1}(S_{k-1}S_{k})
	= X.
\]

We define: $h\colon 0 \mapsto \mathfrak{B}(S_k X)$,
$1 \mapsto \mathfrak{B}(S_{k-1}X)$, where $\mathfrak{B}$ is the
sliding block code of length $2\ell$ from \cref{the:sturmian}
such that $\mathfrak{B}(\infw{s}_{\alpha,\alpha}) = T\partial_{\infw{s},\ell}$ (and $T$ is the shift operator). Notice now that $ \mathfrak{B}(uv) = \mathfrak{B}(u\pref_{2\ell-1}(v))\mathfrak{B}(v)$ for any sufficiently long word $v$ (and $u$ non-empty). Therefore
\begin{align*}
h(\infw{s}_{\beta,\beta})
&=
\mathfrak{B}(S_k X)^{a_{k+1}}
\mathfrak{B}(S_{k-1} X)
\mathfrak{B}(S_k X)^{a_{k+1}}
\cdots\\
&=
\mathfrak{B}(S_k^{a_{k+1}}S_{k-1}S_k^{a_{k+1}}\cdots)
= \mathfrak{B}(\infw{s}_{\alpha,\alpha}) = T\partial_{\infw{s},\ell}.
\end{align*}
\end{proof}

We illustrate the above construction with a couple of examples for the benefit of the interested reader.

\begin{example}
\label{ex:morphicImage}
Take the characteristic Sturmian word $\infw{s}$ of slope
$\alpha = 1 - 1/\sqrt{3}$. The continued fraction expansion
of $\alpha$ is $[0;2,2,1,2,1,2,1,2,1,\ldots]$. Let $\ell = 2$. Then we
have $S_{-1} = 1$, $S_0 = 0$, $S_1 = 01$, $S_2 = 01010$, $S_3 = 0101001$, \ldots.
Here $|S_2S_1| = 7 \geq 5 = 2\ell + 1$. We have $X = \pref_{3}(S_2S_1) = 010$.
Then, defining $h\colon 0 \mapsto \mathfrak{B}(01010X) = 01234$, $1 \mapsto \mathfrak{B}(01X) = 01$,
we find $T\partial_{\infw{s},\ell} = h(\infw{s})$ (see \cref{lem:distinctSets}).
Similarly, for $\ell = 3$ we have $|S_2S_1| = 2\ell + 1$. Then $T\partial_{\infw{s},\ell} = h(\infw{s})$ when $h$ is defined by
$0 \mapsto \mathfrak{B}(01010X) = 01234$ and $1 \mapsto \mathfrak{B}(01X) = 56$, where $X = \pref_{5}(S_2S_1) = S_2 = 01010$.
Let then finally $\ell = 4$. Now we have
$|S_2S_1| = 7 < 9 = 2\ell+1$, but $|S_3S_2| = 12$. Hence
$X = 0101001 = S_3$, and
$T\partial_{\infw{s},\ell} = h(\infw{s}')$, where $h$ is the
morphism defined by
$0 \mapsto \mathfrak{B}(0101001X) = 0123456$ and
$1 \mapsto \mathfrak{B}(01010X) = 01278$,
and $\infw{s}'$ is the characteristic Sturmian word whose slope $\beta$ has continued fraction expansion
$[0;3,1,2,1,2,1,2,1,2,\ldots]$. One can verify that
$\beta = 2-\sqrt{3}$.
\end{example}

\begin{example}
\label{ex:morphicImagePi}
The continued fraction expansion of
$\alpha = \frac{1}{2}(\pi-3)$
begins with
\[
[0;14,7,1,586,3,1,2,1,1,\ldots].
\]
The construction in \cref{prop:another-charact} hence gives
$\partial_{\infw{s}_{\alpha,\alpha},2} = h(\infw{s}_{\beta,\beta})$, where $\beta = \frac{2}{\pi-3} - 14$
has continued fraction expansion $[0;7,1,586,3,1,2,1,1,\ldots]$, and $h$ is defined by $0\mapsto 0^{10}1234$, $1 \mapsto 0$.
\end{example}

\begin{example}
Take the slope $\alpha=(3-\sqrt{5})/2$; its continued fraction expansion is $[0;2,1,1,1,\ldots]$.
Using the previous notation, $S_{-1} = 1$, $S_0 = 0$, and $S_{k+1} = S_k S_{k-1}$ for all  $k\geq 0$.
Then the sequence $(S_k)_{k\geq 0}$ converges to the Fibonacci word; the first few words in the sequence $(S_k)_{k\ge 0}$ are $0, 01, 010, 01001, 01001010$.

Now for any $\ell \geq 2$, the above proposition thus gives that
$\partial_{\infw{f},\ell}$ is the morphic image of the characteristic
Sturmian word of slope
$\beta = \alpha$. In other words, the $\ell$-boundary sequence is
always a morphic image of $\infw{f}$.
\end{example}

We generalize the last observation made in the above example.
\begin{corollary}
\label{cor:morphicSturmian} 
  Let $\mathbf{s}$ be a Sturmian word with quadratic slope.
  Then $\partial_{\mathbf{s},\ell}$ is morphic. In particular,
  the $\ell$-boundary sequence of a Sturmian word fixed by a non-trivial morphism is morphic.
\end{corollary}
\begin{proof}
A remarkable result of Yasutomi \cite{Yasutomi1999sturmian} (see also \cite{BertheEIR2007substitution}), characterizing those Sturmian words that
are fixed by some non-trivial morphism,
implies that if a Sturmian word of slope~$\alpha$ is fixed by a
non-trivial morphism, then so is the characteristic Sturmian word of
slope~$\alpha$. Furthermore, the slope is characterized by the property that $\alpha = [0;1,a_2,\overline{a_3,\ldots,a_r}]$ with $a_r \geq a_2$
or $\alpha = [0;1+a_1,\overline{a_2,\ldots,a_r}]$ with
$a_r \geq a_1 \geq 1$ \cite{CrispMPS1993substitution,Parvaix1997Proprietes} (see also
\cite[Thm.~2.3.25]{Lothaire}). Here $\overline{x_1,\ldots,x_t}$
indicates the periodic tail of the infinite continued fraction
expansion. As $\alpha$ is quadratic, it has an eventually periodic
continued fraction expansion. There thus exist arbitrarily large $k$ for
which $\beta = [0;a_k+1,a_{k+1},\ldots]$ gives a characteristic Sturmian word of slope $\beta$ which is the fixed point of a non-trivial morphism (it is of the latter form). \cref{prop:another-charact} then posits that
$T\partial_{\infw{s},\ell}$ is the morphic image of this word, and the claim follows (because prepending the letter $\partial_{\infw{s},\ell}[\ell]$ preserves morphicity \cite[Thm.~7.6.3]{AS}).
\end{proof}

Notice that given the morphism fixing a Sturmian word
$\infw{s}$, one can compute (the continued fraction expansion of) the quadratic slope (and intercept) of $\infw{s}$ \cite{TanW2003invertible,PengT2011Sturmian,LepsovaPS2022faithful}.
Furthermore,  any (not necessarily pure) morphic Sturmian word has quadratic slope~\cite{AlloucheCassaigneShallitZamboni2017,BertheHoltonZamboni2006},  so in particular the boundary sequence of such a word is morphic.

The above corollary has an alternative proof via the logical
approach as well. For the definitions of notions that follow, we
refer to the cited papers. From the work of Hieronymi and Terry
\cite{HieronymiT2018Ostrowski}, it is known that addition in
the \emph{Ostrowski-numeration system} based on an irrational
quadratic number $\alpha$ is recognizable by a finite
automaton. This motivated Baranwal, Schaeffer, and Shallit to
introduce \emph{Ostrowski-automatic sequences}
in \cite{BaranwalSS2021Ostrowski}. For example, they showed
that the characteristic Sturmian word of slope $\alpha$ is
Ostrowski $\alpha$-automatic. Since the numeration system is
addable, the above corollary follows by the same arguments as
in \cref{ss:logic}.

 We remark that it is unclear to
us whether some of the results proved in this section could be proved
automatically using the very recent tool {\tt Pecan} developed in \cite{OeiMSH2021pecan,HieronymiMOSSS2022decidability}.

\subsection{Factor complexities of the extended boundary sequences}
\label{subsec:complexities}

\begin{definition}
A word over an alphabet $A$ is of {\em minimal complexity}
if its factor complexity is $n + |A|-1$ for all $n \geq 1$.
\end{definition}
Minimal complexity words can be seen as a generalization of
Sturmian words to larger alphabets: if a word (containing all
letters of $A$) has less than $n + |A| - 1$ factors of length
$n$ for some $n$, then it is ultimately periodic. Otherwise it
is aperiodic (a consequence of the Morse--Hedlund theorem).
See \cite{Paul1974minimal,Coven1974sequences,FerencziM1997transcendence,Cassaigne1997sequences,Didier1999characterisation} for characterizations and generalizations.

The following proposition is almost immediate after the key \cref{lem:distinctSets}.

\begin{proposition}
\label{pro:minimal-complexity}
Let $\ell \geq 2$. The $\ell$-boundary sequence of a Sturmian
word is a minimal complexity word (of complexity $n \mapsto n + 2\ell$, $n\geq 1$).
\end{proposition}
\begin{proof}
Recall that $\partial_{\infw{s},\ell}$ is obtained by a coding
of the $2\ell$-block coding of $\infw{s}_{\alpha,\alpha}$.
The following lemma says that the coding is actually a bijection;
in other words, a length-$2\ell$ factor of $\infw{s}$ uniquely
determines a boundary set, or a letter, in the boundary sequence. We conclude that the factors
of length $n + 2\ell - 1$ of $\infw{s}$ uniquely determine
a factor of length $n$ in the $\ell$-boundary sequence.
Since there are $n + 2\ell$ such factors of $\infw{s}$,
the claim follows as the number of factors of length $2\ell$ of 
$\infw{s}$, that is, the number of letters in $\partial_{\infw{s},\ell}$, is $2\ell + 1$.
\end{proof}

\begin{lemma}
\label{lem:distinctSets}
Let $\sigma$ and $\sigma' \in T_{\ell}$, $\ell \geq 2$, be distinct mappings both satisfying \eqref{eq:realizable-constellations} (for different $n$).
Then $\partial_{\sigma} \neq \partial_{\sigma'}$.
\end{lemma}
\begin{proof}
Let $\sigma$ (resp., $\sigma'$) satisfy \eqref{eq:realizable-constellations} with $n$ (resp., $m$ in place of $n$, $m\neq n$). Since
$\sigma \neq \sigma'$, there exist $j \in \{0,\ldots,\ell\}$,
and distinct factors $v$, $v'\in \Fac_{\ell}(\infw{s})$
such that $R_{\alpha}^n(\{-j\alpha\}) \in I_{v}$ and
$R_{\alpha}^m(\{-j\alpha\}) \in I_{v'}$.
To fix a rotation direction, assume without loss of generality that $v'$ is lexicographically less than $v$, so $I_{v'}$ appears before $I_{v}$ in clockwise order, starting from $0$, in the $1$-dimensional torus $\mathbb{T}$.
The situation is depicted in \cref{fig:rotmincompl}: the interval $I_v$ (resp., $I_{v'}$) is colored in orange (resp., dark red).
\begin{figure}[h!t]
\centering
\includegraphics[width=.5\linewidth]{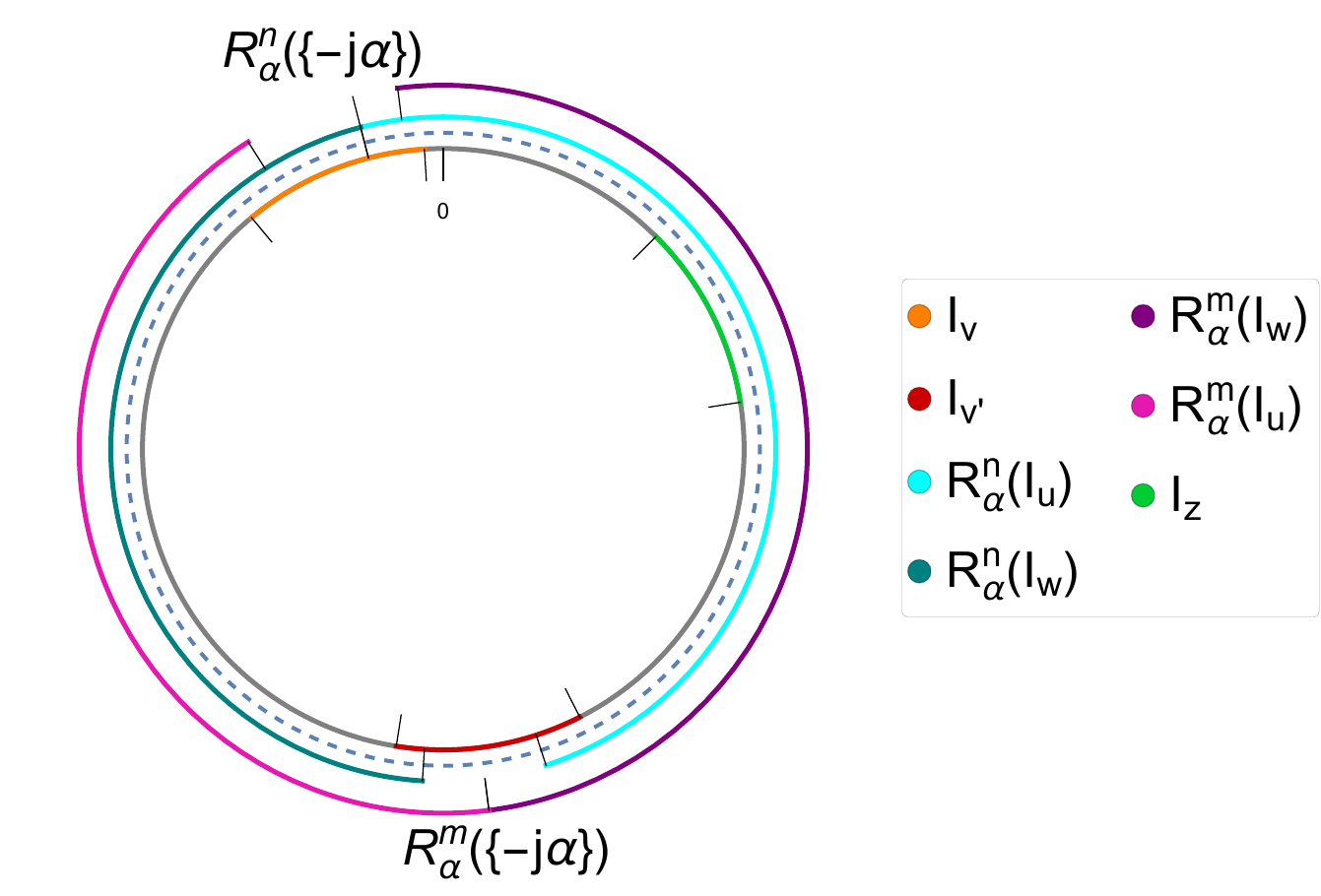}
\caption{The situation depicted in the proof of \cref{lem:distinctSets}.}
 \label{fig:rotmincompl}
\end{figure}
Say that $\{-j\alpha\}$ is the starting point (in clockwise direction) of the interval $I_{u}$, and is the ending point of the interval $I_{w}$ (again in clockwise direction); in particular, $I_u$ and $I_w$ are adjacent intervals.
In particular, in \cref{fig:rotmincompl}, the interval $R_{\alpha}^n(I_{u})$ in light turquoise (resp., $R_{\alpha}^m(I_{u})$ in pink)  appears after the interval $R_{\alpha}^n(I_w)$ in dark turquoise (resp., $R_{\alpha}^m(I_{w})$ in purple) in clockwise order.
We now have that $\partial_{\sigma}$ contains $(u,v)$ and
$(w,v)$, while $\partial_{\sigma'}$ contains $(u,v')$ and
$(w,v')$.
Assume towards a contradiction, that $\partial_{\sigma} = \partial_{\sigma'}$.
Then we must have $(u,v')$ and $(w,v') \in \partial_{\sigma}$ as well as $(u,v)$ and $(w,v) \in \partial_{\sigma'}$.
We have the following:
$R_{\alpha}^n(I_{u}) \cap I_{v} \neq \emptyset \neq R_{\alpha}^n(I_{u}) \cap I_{v'}$ (this is shown in \cref{fig:rotmincompl} where the light turquoise interval intersects both the orange and dark red interval)
and similarly
$R_{\alpha}^n(I_{w}) \cap I_{v} \neq \emptyset \neq R_{\alpha}^n(I_{w}) \cap I_{v'}$ (this is shown in \cref{fig:rotmincompl} where the dark turquoise interval intersects both the orange and dark red interval).
Since $I_u$ and $I_w$ are intervals, we see that $R^n(I_u)$
covers all intervals $I_z$ between $I_v$ and $I_v'$ in
clockwise order starting from the point
$R_{\alpha}^n(\{-j\alpha\})$.
Again, this is illustrated in \cref{fig:rotmincompl} where an interval $I_z$ is depicted in green.
Similarly $R_{\alpha}^n(I_w)$ contains all intervals $I_z$ between $I_v$ and $I_v'$ in anticlockwise order starting from the point
$R_{\alpha}^n(\{-j\alpha\})$. The total number of the
intermediate intervals $I_z$ is $\ell + 1 - 2 = \ell -1\geq 1$,
so assume without loss of generality that $R_{\alpha}^n(I_u)$
covers the interval $I_z$. In particular, this means
that $(u,z) \in \partial_{\sigma}$. But, we have a symmetric situation as follows:
the interval $R_{\alpha}^m(I_u)$ covers all intervals between
$I_{v'}$ and $I_{v}$ in clockwise order starting from $R_{\alpha}^m(\{-j\alpha\})$: these are the same intervals
covered by $R^n(I_w)$. Since $(w,z) \notin \partial_{\sigma}$,
we get the contradiction that $(u,z) \notin \partial_{\sigma'}$.
This suffices for the claim.
\end{proof}

We conclude with a formula for the factor complexity of the
$1$-boundary sequence of Sturmian words.

\begin{proposition}
Let $r$ be the maximal integer such that
$(01)^r$ appears in the Sturmian word $\infw{s}$.
The boundary sequence
$\partial_{\infw{s}}$ has factor complexity
\begin{equation*}
n \mapsto
\begin{cases}
n+1, &\text{if }n < 2r;\\
n+2, & \text{ otherwise}.
\end{cases}
\end{equation*}
\end{proposition}

\begin{proof}
Without loss of generality, we assume that $00$ appears in $\infw{s}$ and $11$ does not.
Let $\mathfrak{B}$ be the length-$2$ sliding block code from \cref{the:sturmian}; it is not hard to show that $\mathfrak{B}$ is defined by $(00) \mapsto 0$, $(01),(10)\mapsto 1$.
To prove the claim, we show that $\mathfrak{B}(u) = \mathfrak{B}(v)$ with $u \neq v$ if and only if $u$ is a prefix of $(01)^r$ and $v$ is a prefix of $(10)^r$ (assuming $|u|,|v|\ge 2$).
This is enough since, as in the proof of \cref{pro:minimal-complexity}, a factor of length $n + 1$ of $\infw{s}$ corresponds to a factor of length $n$ of $\partial_{\infw{s}}$.

Observe that if $u$ is a prefix of $(01)^r$ and $v$ is a prefix of $(10)^r$, then $\mathfrak{B}(u) = \mathfrak{B}(v) = 1^{|u|-1}$.
Let us show the converse by induction on the length of $u,v$, and hence assume that $\mathfrak{B}(u) = \mathfrak{B}(v)$ with $u \neq v$.
If $|u|=2=|v|$, the claim is clear.
Assume then that $|u|,|v|>2$.
If $u$ and $v$ begin with the same letter, then their second letter must be equal, because otherwise $\mathfrak{B}(u)$ begins with $0$ and $\mathfrak{B}(v)$ with $1$ or vice versa.
So write $u=abu'$ and $v=abv'$ for some letters $a,b\in \{0,1\}$ and some binary words $u',v'$.
Since the words $bu'$ and $bv'$ are shorter and distinct, and have equal $\mathfrak{B}$-images, the induction hypothesis implies that one is a prefix of $(01)^r$ and the other a prefix of $(10)^r$.
This is, of course, impossible.
We conclude that the words $u$ and $v$ begin with distinct letters.
Without loss of generality, suppose that $u$ begins with $0$ and $v$ with $1$.
Since $11$ does not appear in $\infw{s}$, we deduce that $v$ begins with $10$, hence $\mathfrak{B}(v)$ begins with $1$.
Therefore $u$ must begin with $01$ for $\mathfrak{B}(u)$ to begin with $1$.
Removing the first letter of $u$ and $v$ allows us to use induction to complete the claim.
\end{proof}

As an immediate corollary, we see that the $\ell$-boundary
sequence is aperiodic for all $\ell \geq 1$.

\section{Conclusions}
There is no particular reason to consider boundary pairs of equal length. One may just as well
define the $(k,\ell)$-boundary sequence in an analogous manner. All the results appearing in
\cref{ss:logic,sec:automatic}
can be extended straightforwardly to account for this seemingly more general notion.
The methods used in \cref{sec:Sturmian} can also be adapted to deal with $(k,\ell)$-boundary
sequences straightforwardly.

\section*{Acknowledgments}
We thank Jean-Paul Allouche for references \cite{Coven1974sequences,Parvaix1997Proprietes,Paul1974minimal}, and Jeffrey Shallit for
discussions about the ``logical approach''.
The anonymous referees are warmly thanked for providing useful feedback improving the quality of the text.


 \bibliographystyle{plainurl}
\bibliography{../biblio}

\end{document}